\documentclass{amsart}
\usepackage{latexsym,amssymb}
\usepackage{bm, color}
\usepackage{graphicx} 

\newcommand{\hF}{\widehat{F}}
\newcommand{\cE}{\mathcal{E}_h}
\newcommand{\cP}{\mathcal{P}}
\newcommand{\cbP}{\bld{\mathcal{P}}}
\newcommand{\cQ}{\mathcal{Q}}

\newcommand{\cS}{\mathcal{S}}
\newcommand{\cT}{\mathcal{T}}

\newcommand{\bfR}{\mathbf{R}}
\newcommand{\bv}{{\bm v}}
\newcommand{\Ome}{\Omega}
\newcommand{\oOme}{\overline{\Omega}}
\renewcommand{\div}{\mbox{\rm div\,}}
\newcommand{\curl}{\mbox{\rm curl\,}}
\newcommand{\tr}{\mbox{\rm tr}}
\newcommand{\mce}{\mathcal{E}_h}
\newcommand{\mct}{\mathcal{T}_h}
\newcommand{\bV}{\bm V}

\newcommand{\p}{\partial}
\newcommand{\Div}{{\rm div}}
\newcommand{\nab}{\nabla}
\newcommand{\bH}{\bm H}
\newcommand{\bC}{\bm C}
\newcommand{\bl}{\big\langle}
\newcommand{\br}{\big\rangle}
\newcommand{\Bl}{\Big\langle}
\newcommand{\Br}{\Big\rangle}
\newcommand{\bcurl}{{\bf curl}}
\newcommand{\Del}{\Delta}
\newcommand{\btau}{{\bm \tau}}

\newcommand{\bL}{{\bm L}}

\newcommand{\bP}{{\bm {P}}}

\newcommand{\bq}{\bm q}
\newcommand{\us}{\tilde}
\newcommand{\bA}{\bm A}
\newcommand{\bld}[1]{\boldsymbol{#1}}

\newcommand{\bg}{\bm g}
\newcommand{\bz}{\bm z}
\newcommand{\bmu}{\bm \mu}
\newcommand{\bW}{\bm W}
\newcommand{\pol}{\mathbb{P}}
\newcommand{\ext}{\mbox{ext}}
\newcommand{\hH}{\widehat{H}}

\newcommand{\cV}{\mathcal{V}}
\newcommand{\bcV}{{\boldsymbol{\mathcal V}}}

\numberwithin{equation}{section}

\newtheorem{definition}{Definition}[section]
\newtheorem{remark}{Remark}[section]
\newtheorem{remarks}{Remarks}[section]
\newtheorem{lemma}{Lemma}[section]
\newtheorem{theorem}{Theorem}[section]
\newtheorem{proposition}{Proposition}[section]
\newtheorem{corollary}{Corollary}[section]
\newtheorem{algorithm}{Algorithm}[section]

\begin{document}

\title[FE differential calculus and applications]{Discontinuous Galerkin 
finite element differential calculus and applications to numerical solutions 
of linear and nonlinear partial differential equations}
 
\author[X. Feng]{Xiaobing Feng}
\address{Department of Mathematics \\
         The University of Tennessee \\
         Knoxville, TN 37996.}
\email{xfeng@math.utk.edu}

\author[T. Lewis]{Thomas Lewis}
\address{Department of Mathematics \\
         The University of Tennessee \\
         Knoxville, TN 37996.}
\email{tlewis@math.utk.edu}

\author[M. Neilan]{Michael Neilan}
\address{Department of Mathematics \\
            University of Pittsburgh \\
            Pittsburgh, PA 15260}
\email{neilan@pitt.edu}

\thanks{The work of the first and second authors were partially supported by the NSF
grant DMS-071083. The research of the third author was partially supported by the  NSF grant DMS-1238711.}

\begin{abstract}
This paper develops a discontinuous Galerkin (DG) finite element differential 
calculus theory for approximating weak derivatives of Sobolev functions 
and piecewise Sobolev functions.  By introducing  numerical one-sided 
derivatives as building blocks,  various first and second order numerical 
operators such as the gradient, divergence, Hessian, and Laplacian operator are 
defined, and their corresponding calculus rules are established. Among the calculus rules are
product and chain rules, integration 
by parts formulas and the divergence theorem. Approximation properties and 
the relationship between the proposed DG finite element numerical derivatives 
and some well-known finite difference numerical derivative formulas on 
Cartesian grids are also established.  Efficient implementation of the 
DG finite element numerical differential operators is also proposed.  
Besides independent interest in numerical differentiation,  the primary 
motivation and goal of developing the DG finite element differential 
calculus is  to solve partial differential equations. 
It is shown that several existing finite element, finite difference and 
DG methods can be rewritten compactly using the proposed  
DG finite element differential calculus framework. Moreover,  new DG 
methods for linear and nonlinear PDEs are also obtained from the 
framework.
\end{abstract}

 

\maketitle
\section{Introduction}\label{sec-1}
Numerical differentiation is an old but basic topic in numerical mathematics.  
Compared to the large amount of literature on numerical integration,  numerical 
differentiation is a much less studied topic. Given a differentiable function,  
the available numerical methods for computing its derivatives are indeed very limited.  
There are essentially only two such methods (cf.\,\cite{Stoer_Bulirsch93}). 
One method is to approximate derivatives by difference quotients. 
The other is to first approximate the given function (or its values 
at a set of points) by a more simple function (e.g., polynomial, rational 
function and piecewise polynomial) and then to use the derivative of the 
approximate function as an approximation to the sought-after derivative. 
The two types of classical methods work well if the given function is 
sufficiently smooth.  However,  the two classical methods produce large errors or divergent 
approximations if the given function is rough, which is often the case 
when the function is a solution of a linear or nonlinear partial differential 
equation (PDE).  

For boundary value and initial-boundary value problems, 
classical solutions often do not exist. Consequently, one has to deal with 
generalized or weak solutions, which are defined using a variational setting 
for linear and quasilinear PDEs.  Although numerical methods for PDEs 
implicitly give rise to methods for approximating weak derivatives 
(in fact, combinations of weak derivatives) of the solution functions 
(cf. \cite{BuffaOrtner09, BurmanErn08, EyckLew06, Wang_Ye11}),  
to the best of our knowledge,  
there is no systematic study and theory in the literature on how to
approximate weak derivatives of a given (not-so-smooth) function.  
Moreover, for linear second order PDEs of non-divergence form and 
fully nonlinear PDEs,  it is not possible to derive variational weak 
formulations using integration by parts.  As a result,  weak solution 
concepts for those types of PDEs are different. The best known and most 
successful one is the {\em viscosity solution} 
concept (cf.\,\cite{Crandall_Ishii_Lions92, Feng_Glowinski_Neilan12} 
and the references therein). To directly approximate viscosity solutions, 
which in general are only continuous functions,  one must approximate 
their derivatives in some appropriately defined sense 
offline (cf.\,\cite{Feng_Lewis12, LakkisPryer12}), and then substitute 
the numerical derivatives for the (formal) derivatives appearing in the PDEs.  
Clearly, to make such an intuitive approach work,  the key is to 
construct ``correct" numerical derivatives and 
to use them judiciously to build numerical schemes. 

This paper addresses the above two fundamental issues. 
The specific goals of this paper are twofold.  
{\em First}, we systematically develop a computational framework 
for approximating weak derivatives and a new discontinuous Galerkin (DG) 
finite element differential calculus theory.  Keeping in mind the 
approximation of fully nonlinear PDEs, we introduce locally defined, 
one-sided numerical derivatives for piecewise weakly differentiable 
functions.  Using the newly defined one-sided numerical derivatives 
as building blocks, we then define a host of first and second 
order sided numerical differential operators including the gradient,  
divergence,  curl, Hessian, and Laplace operators.  To ensure the 
usefulness and consistency of these numerical operators, 
we establish basic calculus rules for them. 
Among the rules are the product and the chain rule,
integration by parts formulas and the divergence theorem.  
We establish some approximation properties of the proposed 
DG finite element numerical derivatives and show that they coincide 
with well-known finite difference  derivative formulas on Cartesian grids.  
Consequently,  our DG finite element numerical derivatives are natural 
generalizations of well-known finite difference numerical derivatives on 
general meshes.  These results are of independent interest 
in numerical differentiation.  
{\em Second},  we present some applications of the proposed 
DG finite element differential calculus to build 
numerical methods for linear and nonlinear partial differential 
equations.  This is done based on a very simple idea; that is,  
we replace the (formal) differential operators in the given PDE by 
their corresponding DG finite element numerical operators and project 
(in the $L^2$ sense) the resulting equation onto the DG 
finite element space $V^h_r$.  We show that the resulting 
numerical methods not only recover several existing finite difference, 
finite element and DG methods, but also give rise 
to some new numerical schemes for both linear and nonlinear PDE problems.  

The remainder of this paper is organized as follows.  
In Section \ref{sec-2} we introduce the mesh and space notation 
used throughout the paper.  In Section \ref{sec-3} we give the 
definitions of our DG finite element numerical derivatives and various 
first and second order numerical differential operators. 
In Section \ref{sec-4} we establish an approximation property 
and various calculus rules for the DG finite element numerical derivatives 
and operators. In Section \ref{sec-5} we discuss the implementation 
aspects of the numerical derivatives and operators. Finally, in 
Section \ref{sec-6} we present several applications of the 
proposed DG finite element differential calculus to 
numerical solutions of prototypical linear and nonlinear PDEs including 
the Poisson equation, the biharmonic equation, the $p$-Laplace 
equation, second order linear elliptic PDEs in non-divergence form,  
first order fully nonlinear Hamilton-Jacobi equations, 
and second order fully nonlinear Monge-Amp\`ere equations. 

\section{Preliminaries}\label{sec-2}

Let $d$ be a positive integer,  $\Omega\subset\mathbf{R}^d$  be a bounded 
open domain, and $\mathcal{T}_h$ denote a locally quasi-uniform and 
shape-regular partition of $\Omega$ \cite{ciarlet78}. Let $\mce^I$ denote the set of all 
interior faces/edges of $\mct$, $\mce^B$ denote the set of all boundary 
faces/edges of $\mct$, and $\mce:=\mce^I\cup \mce^B$.

Let $p\in [1,\infty]$ and $m\geq 0$ be an integer. Define the following 
piecewise $W^{m,p}$ and piecewise $C^m$ spaces with respect to the mesh $\cT_h$:
\[
W^{m,p}(\cT_h):=\prod_{K\in \cT_h} W^{m,p}(K),\qquad
C^m(\cT_h):= \prod_{K\in \cT_h} C^m(\overline{K}).
\]
When $p=2$, we set $H^m(\cT_h): =W^{m,2}(\cT_h)$. 
We also define the analogous piecewise vector-valued spaces
as $\bH^m(\cT_h):=[H^m(\mct)]^d$, $\bW^{m,p}(\mct) = [W^{m,p}(\mct)]^d$,
 $\bC^m(\mct) = [C^m(\mct)]^d$,
and the matrix-valued spaces $\tilde{\bH}^m(\mct):=[H^m(\mct)]^{d\times d}$,
$\tilde{\bW}^{m,p}(\mct):=[W^{m,p}(\mct)]^{d\times d}$,
and $\tilde{\bC}^m(\mct) = [C^m(\mct)]^{d\times d}$.
The piecewise $L^2$-inner product over the mesh $\cT_h$ is given by
\[
(v,w)_{\mathcal{T}_h}:= \sum_{K\in \cT_h} \int_{K} v w\, dx,
\]
and for a set $\mathcal{S}_h \subset \mce$,
the piecewise $L^2$-inner product over $\mathcal{S}_h$ is given by
\begin{align*}
\bl v,w\br_{\mathcal{S}_h} :=\sum_{e\in \mathcal{S}_h} \int_e vw\, ds.
\end{align*}
Angled brackets without subscripts $\bl \cdot,\cdot\br$ represent 
the dual pairing between some Banach space and its dual.

For a fixed integer $r \geq 0$, we define the standard discontinuous 
Galerkin (DG) finite element space
$V^h_r \subset W^{m,p}(\cT_h)\subset L^2 (\Ome)$ by
\[
V^h_r := \prod_{K \in \mathcal{T}_h} \pol_{r} (K),
\]
where $ \pol_{r} (K)$ denotes the set of all polynomials on $K$ with
degree not exceeding $r$.  The analogous vector-valued and matrix valued 
DG spaces are given by $\bV_r^h:=[V_r^h]^d$ and 
$\tilde{\bV}_r^h:=[V_r^h]^{d\times d}$. In addition, we define
\[
\cV_h := W^{1,1}(\mct)\cap C^0(\mct),
\]
$\bcV_h:=[\cV_h]^d$, and $\tilde{\bcV}_h :=[\cV_h]^{d\times d}.$  We note
that $V^h_r\subset \cV_h$.
We denote by $\cP_r^h : L^2(\Omega) \to V_r^h$ the $L^2$ projection operator onto $V_r^h$, 
$\cbP_r^h : \left[ L^2(\Omega) \right]^d \to \bm{V}_r^h$ the $L^2$ projection operator onto $\bm{V}_r^h$ and
$\widetilde{\cbP}_r^h : \left[ L^2(\Omega)\right]^{d \times d} \to \widetilde{\bm{V}}_r^h$ the $L^2$ projection operator onto 
$\widetilde{\bm{V}}_r^h$.

Let $K, K'\in \cT_h$ and $e=\partial K\cap \partial K'$. Without loss of
 generality, we assume that the global labeling number of $K$ is smaller than 
that of $K'$.  We then introduce the following standard jump and average notations
across the face/edge $e$:
\begin{alignat*}{4}
[v] &:= v|_K-v|_{K'} 
\quad &&\mbox{on } e\in \cE^I,\qquad
&&[v] :=v\quad 
&&\mbox{on } e\in \cE^B,\\
\{v\} &:=\frac12\bigl( v|_K +v|_{K'} \bigr) \quad
&&\mbox{on } e\in \cE^I,\qquad
&&\{v\}:=v\quad 
&&\mbox{on } e\in \cE^B
\end{alignat*}
for $v\in \cV_h$. We also define $n_e:=n_K|_e=-n_{K'}|_e$ as the unit normal on $e$.

\section{Definitions of discrete differential operators}\label{sec-3}

Let $v\in \cV_h$.  For $e\in \cE^I$, that is, 
$e=\partial K\cap \partial K'\in \cE^I$ for some $K,K'\in \cT_h$, we write
$n_e=\bigl(n_e^{(1)}, n_e^{(2)}, \ldots, n_e^{(d)} \bigr)^t$ to be
the unit normal of $e$.
We then define the following three trace operators on $e$ in the direction $x_i$:
\begin{align}\label{e2.1}
&\cQ_i^-(v)(x):= \begin{cases}
\displaystyle{\lim_{y\in K\atop y\to x} v(y)} &\qquad\mbox{if }n_e^{(i)}< 0,\\
\displaystyle{\lim_{y\in K'\atop y\to x}v(y)} &\qquad\mbox{if }n_e^{(i)}\geq 0,
                \end{cases} \\
&\cQ_i^+(v)(x):= \begin{cases}
\displaystyle{\lim_{y\in K'\atop y\to x} v(y)}&\qquad\mbox{if } n_e^{(i)}< 0,\\ 
\displaystyle{\lim_{y\in K\atop y\to x} v(y)}&\qquad\mbox{if } n_e^{(i)}\geq 0,
                \end{cases} \label{e2.2} \\
&\cQ_i(v)(x):= \frac12 \Bigl( \cQ_i^-(v)(x) + \cQ_i^+(v)(x) \Bigr) \label{e2.3}
\end{align}
for any $x\in e$ and $i=1,2,\ldots,d$. We note that $\cQ_i^-$ and $\cQ_i^+$
can be regarded respectively as the ``left" and ``right" limit of
$v$ at $x\in e$ in the direction of $x_i$. 
If $e\in \cE^B$, we simply let
\begin{equation}\label{2.4}
\cQ_i^-(v)(x)=\cQ_i^+(v)(x)=\cQ_i(v)(x):= \lim_{y\in \Omega\atop y\to x} v(y) 
\qquad\forall x\in e.
\end{equation}

\begin{remark}\label{QidentityRemark}
On an interior edge $e\in \mathcal{E}_h^I$, we may alternatively write
\begin{align}\label{Qidentity}
\cQ^\pm_i(v) = \{v\}\pm\frac12 {\rm sgn}(n_e^{(i)})[v],\quad \text{where}\quad 
{\rm sgn}(n_e^{(i)}) = 
\left\{
\begin{array}{cc}
1 &\text{ if } n_e^{(i)} \ge 0,\\
-1 & \text{ if }n_e^{(i)} <0.
\end{array}
\right.
\end{align}
\end{remark}

With the help of the trace operators $\cQ_i^-, \cQ_i^+$ and $\cQ_i$ we 
are ready to introduce our discrete partial derivative 
operators $\partial_{h,x_i}^-, \partial_{h,x_i}^+$, 
$\partial_{h,x_i}: \cV_h\to V^h_r$.

\begin{definition}\label{def1}
For any $v\in \cV_h$, we define the discrete partial derivatives
$\partial_{h,x_i}^-  v$, $\partial_{h,x_i}^+ v$, $\partial_{h,x_i} v \in V^h_r$ by 
\begin{align}
\bigl(\partial_{h,x_i}^\pm v,\varphi_h\bigr)_{\cT_h}
&:= \bigl\langle \cQ_i^\pm (v) n^{(i)}, [\varphi_h] \bigr\rangle_{\mce}
-\bigl(v, \partial_{x_i}\varphi_h \bigr)_{\cT_h} \label{e2.5}   \\
&\hskip 1.4in
+  \bigl\langle \gamma^\pm_i  [v],[\varphi_h] \bigr\rangle_{\mce^I}
\qquad\forall \varphi_h\in V^h_r,  \nonumber \\
\partial_{h,x_i} v &:= \frac12\Bigl(\partial_{h,x_i}^- v 
+ \partial_{h,x_i}^+ v\Bigr), \label{e2.7} 
\end{align}
for $i=1,2,\ldots, d$. Here, $\partial_{x_i}$ denotes the usual (weak) 
partial derivative operator in the direction $x_i$, $n^{(i)}$ is the 
piecewise constant function satisfying $n^{(i)}|_e = n^{(i)}_e$
and $\gamma^-_i$ and $\gamma^+_i$ are piecewise constants with respect 
to the set of interior edges.
\end{definition}

In addition, we define the discrete partial derivatives
when  boundary data is provided.
\begin{definition}\label{def1b}
Let $g\in L^1(\p\Ome)$ be given.  Then for any $v\in \cV_h$, we define the 
discrete partial derivatives
$\partial_{h,x_i}^{-,g}  v, \partial_{h,x_i}^{+,g} v,\partial_{h,x_i}^g v
\in V^h_r$ by 
\begin{align}\label{dgpmDef}
\bigl(\partial_{h,x_i}^{\pm,g} v,\varphi_h\bigr)_{\cT_h}
&:= \big(\p_{h,x_i}^\pm v,\varphi_h)_{\mct} +\bl (g-v)n^{(i)},\varphi_h\br_{\mce^B}
\qquad\forall \varphi_h\in V^h_r, \\
\label{dgDef}
\partial_{h,x_i}^g v &:= \frac12\Bigl(\partial_{h,x_i}^{-,g} v 
+ \partial_{h,x_i}^{+,g} v\Bigr).
\end{align}

\end{definition}

\begin{remarks}\label{rem1}\
\begin{enumerate}
\item[(a)] Since every function $v\in \cV_h$ has a well-defined trace 
in $L^1(\p K)$ and every function $\varphi_h\in V^h_r$ has a well-defined 
trace in $L^\infty(\p K)$ for all $K\in \cT_h$, the last term on  
the right-hand side of \eqref{e2.5} is well defined.

\item[(b)] Since $V^h_r$ is a totally discontinuous piecewise polynomial space,
the discrete derivatives $\partial_{h,x_i}^\pm v$
can also be written in their equivalent {\em local} versions:
\begin{align}\label{e2.5a}
\bigl(\partial_{h,x_i}^\pm v,\varphi_h\bigr)_K
&=\bigl\langle \cQ_i^\pm(v) n_K^{(i)}, \varphi_h \bigr\rangle_{\partial K}
-\bigl(v, \partial_{x_i}\varphi_h \bigr)_K \\
&\hskip 0.8in
+\nonumber\sum_{e\subset \partial K\backslash \p \Ome} 
\gamma_{i,e}^\pm \bigl\langle [v],[\varphi_h] \bigr\rangle_e 
\qquad \forall \varphi_h\in \pol_r(K)  
\end{align}
for $i=1,2,\ldots, d$ and $K\in \cT_h$.   Here, $\gamma^\pm_{i,e} = \gamma^\pm_i|_{e}$.
 
\item[(c)] The discrete derivatives 
$\partial_{h,x_i}^- v, \partial_{h,x_i}^+ v$ and $\partial_{h,x_i} v$
can be regarded, respectively, as ``left", ``right", and ``central" 
discrete partial derivatives of $v$ with respect to $x_i$.
The definitions are analogous to the weak derivative definition.

\item[(d)] We note that the discrete one-sided partial derivatives are 
defined for functions in the space $\cV_h$, 
in particular, for functions in the DG finite element
space $V^h_r\subset \cV_h$.  

\item[(e)]
By the identity \eqref{Qidentity}, we have
\begin{align} \nonumber 
\big(\partial_{h,x_i}^\pm v, \varphi_h\big)_{\mct}
& =\Bl \{v\} n^{(i)}, [\varphi_h]\Br_{\mce}
+\Bl \big(\gamma^\pm_i \pm \frac12 |n^{(i)}|\big) [v], [\varphi_h]\Br_{\mce^I}
-\big( v,\p_{x_i} \varphi_h\big)_{\mct}.
\end{align}
Therefore, $\p_{h,x_i}^+ v = \p_{h,x_i}^- v$ provided 
$\gamma^+_{i,e}-\gamma^-_{i,e} = -|n^{(i)}|$ for all $e\in \mathcal{E}_h^I$.
\end{enumerate}
\end{remarks}

\begin{definition}\label{def2}
We define the following first order discrete operators:
\begin{subequations}
\begin{align}\label{e2.8}
\nabla_h^\pm v &:=\bigl(\partial_{h,x_1}^\pm v, \partial_{h,x_2}^\pm v,\cdots,
\partial_{h,x_d}^\pm v \bigr)^t,\\
\nabla_h v &:= \frac12 \Bigl( \nabla_h^- v + \nabla_h^+ v \Bigr),    \label{e2.8b}\\
\div_h^\pm\bv &:= \partial_{h,x_1}^\pm v_1+\partial_{h,x_2}^\pm v_2 
+\cdots + \partial_{h,x_d}^\pm v_d, \label{e2.11} \\
\div_h\bv &:= \frac12\bigl( \div_h^-\bv +\div_h^+\bv \bigr) \label{e2.11b}
\end{align}
for any  $v\in \cV_h$ and  $\bv=(v_1,v_2,\cdots, v_d)^t
\in \bcV_h$.
In addition, we define the discrete curl operators for $\bv\in \bcV_h$ 
and $v\in \cV_h$:
\begin{align} \label{e2.14}
\bcurl_h^\pm \bv&:=\Big(\p_{h,x_2}^\pm v_3-\p^\pm_{h,x_3} v_2,
\p_{h,x_3}^\pm v_1-\p_{h,x_1}^\pm v_3,\p_{h,x_1}^\pm v_2-\p_{h,x_2}^\pm v_1\Big)^t
\end{align}
when $d=3$,  and 
\begin{align} \label{e2.15}
\curl_h^\pm \bv := \p_{h,x_1}^\pm v_2 -  \p_{h,x_2}^\pm v_1,  \quad
\bcurl_h^\pm v :=\Big(\p_{h,x_2}^\pm v,  -\p_{h,x_1}^\pm v \Big)^t 
\end{align}
when $d=2$.  We also set 
\begin{align} \label{e2.16}
\bcurl_h \bv := \frac12\Bigl( \bcurl_h^- \bv + \bcurl_h^+ \bv \Bigr), \quad
\curl_h \bv = \frac12\bigl( \curl_h^- \bv + \curl_h^+ \bv \bigr).
\end{align}
\end{subequations}
When boundary data $g\in L^1(\p\Ome)$ is provided, 
the analogous operators are given by 
\begin{subequations}
\label{FirstOrderGDef}
\begin{align} \label{gNabDef_pm}
\nabla_{h,g}^{\pm} v &:=\bigl(\partial_{h,x_1}^{\pm,g} v, \partial_{h,x_2}^{\pm,g} v,\cdots, 
\partial_{h,x_d}^{\pm,g} v \bigr)^t, \\
\nabla_{h,g} v &:= \frac12 \Bigl( \nabla_h^{-,g} v + \nabla_h^{+,g} v \Bigr). \label{gNabDef}
\end{align}
In addition, for given $\bg=(g_1,g_2,\ldots,g_d)^t\in \bL^1(\p\Ome)$, we set
\begin{align}
\Div_{h,\bg}^\pm \bv &:= \p_{h,x_1}^{\pm,g_1} v_1+\p_{h,x_2}^{\pm,g_2}v_2
+\cdots+\p_{h,x_d}^{\pm,g_d}v_d,\\
\Div_{h,\bg} \bv &:=\frac12 \big(\Div_{h,\bg}^- \bv +\Div_{h,\bg}^+ \bv\big).
\end{align}

\end{subequations}

\end{definition}

\begin{definition}\label{def3}
Denote by $D_h^\pm$ and $D_h$ the transposes
of the operators $\nab_h^\pm$ and $\nab_h$, respectively; that is,
$D_h^\pm v = (\nabla^\pm_h v)^t$ and $D_h v = (\nabla_h v)^t$.
We then define the following second order discrete operators:
\begin{subequations} \label{SecondOrderDef}
\begin{alignat}{2}  
&D_h^{-\pm} v := D_h^-\nabla_h^\pm v,  &&\qquad D_h^{+\pm} v := D_h^+\nabla_h^\pm v, \label{s6.6_hessians}  \\
&\Delta_h^{-\pm} v := \tr\bigl(D^{-\pm}_h v),  &&\qquad \Delta_h^{+\pm} v := \tr\bigl(D^{+\pm}_h v),\\
&D_h^2 v :=D_h \nabla_h v,  &&\qquad \Delta_h v :=   \tr\bigl( D_h^2 v \bigr)
\end{alignat}
\end{subequations}
for any  $v\in \cV_h$. Here, $\tr(\cdot)$ denotes the
matrix trace operator.  When boundary data $g\in L^1(\p\Ome)$ is given, we define
\begin{subequations}
\begin{alignat}{3}
&D_{h,g}^{-\pm} v := D_h^-\nabla_{h,g}^\pm v,\qquad
&&\qquad D_{h,g}^{+\pm} v := D_h^+\nabla_{h,g}^\pm v, \\
&\Delta_{h,g}^{-\pm} v := \tr\bigl(D_{h,g}^{-\pm} v\bigr), \qquad
&&\qquad \Delta_{h,g}^{+\pm} v := \tr (D_{h,g}^{+\pm}v), \\
&D_{h,g}^2 v := D_h \nabla_{h,g} v ,  
&&\qquad \Delta_{h,g} v :=  \tr(D^2_{h,g} v).
\end{alignat}
\end{subequations}

\end{definition}

\begin{remarks}\label{rem2}\
\begin{enumerate}
\item[(a)] The matrix-valued functions $D_h^{--}v, D_h^{++}v, D_h^{-+}v$ and $D_h^{+-}v$ 
define four copies of 
discrete $d\times d$ Hessian matrices of the function $v$. 
Similarly, $\Delta_h^{--} v, \Delta_h^{++} v, \Delta_h^{-+} v$ 
and $\Delta_h^{+-} v$ are four copies of discrete Laplacians of $v$.

\item[(b)] Since $\nabla_h^- v,\, \nabla_h^+ v \in \bV^h_r\subset
\bcV_h$,  all the above second order discrete operators are well-defined.

\item[(c)] It is simple to see that the definitions of the discrete Laplacians given in Definition {\rm \ref{def3}}
are equivalent to $\Del_h^{-\pm} v = \Div_h^- \nab_h^\pm v$, $\Del_h^{+\pm} v = \Div_h^+ \nab_h^\pm v$
and $\Del_h v = \Div_h \nab_h v$.
\end{enumerate}

\end{remarks}

\section{Properties of discrete differential operators}\label{sec-4}
In this section we shall establish a number of properties for the DG 
finite element (FE) discrete derivatives defined in Definitions \ref{def1} and \ref{def2}. 
We start with some characterization results for the DG FE discrete 
derivatives.  We then establish an approximation property followed by several
basic calculus rules such as product and chain rules, the integration 
by parts formula, the divergence theorem and the relationship of the 
DG FE discrete derivatives with
standard finite difference derivative formulas. 

\subsection{Characterization of DG finite element derivatives}\label{sec-4.0}

For any $v\in H^1(\Omega)$, we have $\cQ_i^\pm(v)|_e = v|_e$ and
$[v]=0$ on $e\in \cE^I$.  Hence,
\[
\bigl(\partial_{h,x_i}^\pm v,\varphi_h\bigr)_{\cT_h} = \bl  v \, n^{(i)}, [\varphi_h]\br_{\mce}
-\bigl(v, \partial_{x_i}\varphi_h \bigr)_{\cT_h} \qquad\forall \varphi_h\in V^h_r.
\]
Integration by parts element-wise immediately yields 
\begin{align}\label{e4.1}
\bigl(\partial_{h,x_i}^\pm v,\varphi_h\bigr)_{\cT_h}
&=\bigl(\partial_{x_i} v,\varphi_h\bigr)_{\cT_h}
\qquad\forall \varphi_h\in V^h_r,
\intertext{and therefore}
\label{e4.1a}
\bigl(\partial_{h,x_i}^\pm v,\varphi_h\bigr)_K
&=\bigl(\partial_{x_i} v,\varphi_h\bigr)_K
\qquad\forall \varphi_h\in \pol_r(K), \, \forall K\in \cT_h. 
\end{align}
Hence, we have the following proposition.

\begin{proposition}\label{prop0a}
For any $v\in \cV_h\cap H^1(\Omega)$, $\partial_{h,x_i}^\pm v$ and 
$\partial_{h,x_i} v$ coincide with the $L^2$-projection of $\partial_{x_i} v$ 
onto $V^h_r$. We write $\partial_{h,x_i}^\pm v=\partial_{h,x_i} v
=\mathcal{P}^h_r \partial_{x_i} v$, 
where $\cP_r^h$ denotes the $L^2$ projection onto $V_r^h$.
\end{proposition}

From the above proposition, we easily derive the following corollary.
\begin{corollary}\label{cor4.1}
Let $v_h\in V^h_\ell\cap H^1(\Ome)$ with $0\le  \ell \le r+1$.
Then $\p^-_{h,x_i} v_h = \p^+_{h,x_i} v_h
=\p_{h,x_i} v_h = \p_{x_i} v_h \in V_{\ell-1}^h$, 
where $V_{-1}^h := \{ 0 \}$, the set with only the zero function.
\end{corollary}

For an arbitrary piecewise 
function $v\in \cV_h$, 
the above characterization results are not expected to hold in general
because $\partial_{x_i} v$ may not exist. Below we shall derive  
a similar characterization to \eqref{e4.1} for $\partial_{h,x_i}^\pm v$ 
and $\partial_{h,x_i} v$ when $v$ is an arbitrary function in $\cV_h$.
 
For any $v\in \cV_h$, let $\mathcal{D}_{x_i} v$ denote the distributional 
derivative of $v$ with respect to $x_i$ (which does exist). Let $\Xi
\subset \Omega$ be a $(d-1)$-dimensional continuous and bounded 
surface.  We define the {\em delta function} $\delta(\Xi, g, x)$
 of variable strength supported on $\Xi$ by 
(cf. \cite{Engquist_Tornberg_Tsai05})
\begin{align}\label{e4.0a}
\bigl\langle \delta (\Xi, g, \cdot), \varphi \bigr\rangle
:= \int_\Xi g(s) \varphi(x(s))\, ds  
\qquad \forall \varphi\in C^0(\Omega),
\end{align}
where $x(s) \in \Xi$.  We also extend the above definition to 
test functions from $V^h_r$ as follows 
\begin{align}\label{e4.0b}
\bigl\langle \delta (\Xi, g, \cdot), \varphi_h \bigr\rangle
:= \int_\Xi g(s) \bigl\{\varphi_h(x(s))\bigr\}\, ds  
\qquad \forall \varphi_h\in V^h_r. 
\end{align}

Using $\delta(\Xi, g, x)$ we give the following characterization 
for $\mathcal{D}_{x_i} v$.
\begin{lemma}\label{lem4.0c}
For every $v\in \cV_h$ there holds the following representation formula:
\begin{align}\label{e4.1b}
\mathcal{D}_{x_i} v= \sum _{K\in \cT_h} \p_{x_i} v\, \chi_K 
- \sum_{e\in \cE^I} n_e^{(i)} \delta(e, [v], x) \qquad \mbox{for a.e. } x\in \Omega,
\end{align} 
where $\chi_K$ denotes the characteristic function supported on $K$. 
\end{lemma}

\begin{proof}
By the definition of $\mathcal{D}_{x_i} v$ and integration by
parts on each $K\in \Omega$,  we have for any $\varphi\in C^1(\Omega)\cap H^1_0(\Ome)$
\begin{align*}
\bigl\langle \mathcal{D}_{x_i} v, \varphi \bigr\rangle
&:=-\bigl\langle v, \p_{x_i} \varphi \bigr\rangle
=-\bigl( v, \p_{x_i} \varphi \bigr)_{\cT_h} \\
&=-\sum_{K\in \cT_h} \bigl\langle v, \varphi n_K^{(i)} \bigr\rangle_{\p K} 
  + \bigl( \p_{x_i} v, \varphi \bigr)_{\cT_h} \\
&= -\sum_{e\in \cE^I}  \bigl\langle [v], \varphi n_e^{(i)} \bigr\rangle_e
  + \bigl( \p_{x_i} v, \varphi \bigr)_{\cT_h} \\
&= -\Bigl\langle \sum_{e\in\cE^I} n_e^{(i)}\delta(e, [v], \cdot),
\varphi\Bigr\rangle
  + \Bigl\langle \sum_{K\in \cT_h} \p_{x_i} v\chi_K, \varphi \Bigr\rangle. 
\end{align*}
Here we have used the fact that $v\in W^{1,1}(K)$ for every $K\in \cT_h$. 
Clearly the above identity infers \eqref{e4.1b}. The proof is complete.
\end{proof}

\begin{proposition}\label{prop0b} 
Set $\gamma^\pm =0$ in \eqref{e2.5}. For any $v\in \cV_h$, 
$\partial_{h,x_i} v$ coincides with the $L^2$-projection of 
$\mathcal{D}_{x_i} v$ onto $V^h_r$ in the sense that
\begin{align}\label{e4.1e}
\bigl( \p_{h,x_i} v, \varphi_h \bigr)_{\cT_h} 
= \bigl\langle \mathcal{D}_{x_i} v, \varphi_h \bigr\rangle 
\qquad \forall \varphi_h\in V^h_r,
\end{align}
where the right hand-side is understood according to \eqref{e4.0b}.
We write $\partial_{h,x_i} v=\mathcal{P}^h_r \mathcal{D}_{x_i} v$. 
\end{proposition}

\begin{proof}
By the definition of $\p_{h,x_i} v$ and using the 
integration by parts formula for piecewise functions, 
we get, for $i=1,2,\ldots, d$, 
\begin{align*} 
\bigl(\partial_{h,x_i} v,\varphi_h\bigr)_{\cT_h}
&= \bigl\langle \{v\} n^{(i)}, [\varphi_h] \bigr\rangle_{\mce}
-\bigl(v, \partial_{x_i}\varphi_h \bigr)_{\cT_h} \\
&= -\bigl\langle [v] n^{(i)}, \{\varphi_h\} \bigr\rangle_{\mce^I} 
  + \bigl(\partial_{x_i} v, \varphi_h \bigr)_{\cT_h} \nonumber\\
&=- \Bigl\langle \sum_{e\in \cE^I} n^{(i)} \delta(e, [v], \cdot), 
   \varphi_h \Bigr\rangle + \Bigl\langle \sum_{K\in\cT_h} \p_{x_i} v\chi_K,
    \varphi_h \Bigr\rangle \nonumber \\
&= \bigl\langle \mathcal{D}_{x_i} v, \varphi_h \bigr\rangle
\qquad\forall \varphi_h\in V^h_r.  \nonumber
\end{align*}
Here we have used \eqref{e4.0b} and \eqref{e4.1b} to get the 
final two equalities. The proof is complete.
\end{proof} 

\begin{remark} 
Define the ``lifting operator" $\mathcal{L}_{h,i}: L^1(\cE)\to V^h_r$ by 
\begin{equation}\label{lift}
\bigl(\mathcal{L}_{h,i} v,\varphi_h\bigr)_{\cT_h}
:=\Bigl\langle \sum_{e\in \cE^I} n^{(i)} \delta(e, [v], \cdot), 
   \varphi_h \Bigr\rangle
=\bigl\langle [v] n^{(i)}, \{\varphi_h\} \bigr\rangle_{\mce^I}.
\end{equation}
Then we have
\begin{equation}\label{lift_1}
\partial_{h,x_i} v = \sum_{K\in\cT_h} \p_{x_i} v\chi_K - \mathcal{L}_{h,i} v.  
\end{equation}

\end{remark}

Finally, combining Propositions \ref{prop0a} and \ref{prop0b}
and the well-known limiting characterization theorem of distributional 
derivatives (cf. \cite[Theorem 6.32]{Rudin91}), we obtain another
characterization for our DG FE derivatives.

\begin{proposition}\label{prop0c}
Set $\gamma^\pm =0$ in \eqref{e2.5}. Then for any $v\in \cV_h$,
there exists a sequence of functions $\{v_j\}_{j\geq 1} \subset 
C^\infty_0(\Omega)$ such that, for $i=1,2,\cdots, d$,
\begin{enumerate}
\item[{\rm (i)}] $v_j \to  v$ as $j\to \infty$ in $L^1(\Omega)$.
\item[{\rm (ii)}] $\p_{h,x_i} v_j\to \mathcal{P}^h_r 
\mathcal{D}_{x_i} v$
as $j\to \infty$ weakly in $V^h_r$ in the sense that 
\begin{equation}\label{e4.1g}
\lim_{j\to \infty} \bigl( \p_{h,x_i} v_j, \varphi_h \bigr)_{\cT_h} 
=\bigl(\mathcal{P}^h_r \mathcal{D}_{x_i} v, \varphi_h \bigr)_{\cT_h} 
\qquad\forall \varphi_h\in V^h_r.
\end{equation}
\end{enumerate}
 
\end{proposition}

\begin{proof}
Let $\{\rho_j\}_{j\geq 1}$ denote a symmetric mollifier (or 
approximate identity) with compact support. For any 
$v\in \cV_h\subset L^1(\Omega)$, set 
$v_j:=v\ast \rho_j\in C_0^\infty(\Omega)$, the convolution of $v$ 
and $\rho_j$.  Then it is easy to check that the sequence $\{v_j\}$ 
fulfills the properties (i) and (ii). The proof is complete. 
\end{proof}

\subsection{Approximation properties}\label{sec-4.1}
By Proposition \ref{prop0a} we readily have the following 
approximation properties for the DG FE derivatives.

\begin{theorem}\label{lem4.1}
For any $v\in \cV_h\cap H^1(\Ome)$ there holds the following inequalities:
\begin{alignat}{2}\label{e4.2}
\|\partial_{x_i} v-\partial_{h,x_i}^* v\|_{L^2(\cT_h)}
&\leq \|\partial_{x_i} v-\partial_{x_i} \psi_h\|_{L^2(\cT_h)}
\qquad&&\forall \psi_h\in V^h_{r+1},\\
\|\partial_{x_i} v-\partial_{h,x_i}^* v\|_{L^2(K)}
&\leq \|\partial_{x_i} v-\partial_{x_i} \psi_h\|_{L^2(K)}
\qquad&&\forall \psi_h\in \pol_{r+1}, \label{e4.2a} 
\end{alignat}
where $*$ can take $+,-$ and empty value.
\end{theorem}
\begin{proof}
Notice that \eqref{e4.1} can be rewritten as 
\[
\bigl(\partial_{x_i} v-\partial_{h,x_i}^* v,\varphi_h\bigr)_{\cT_h}
=0 \qquad\forall \varphi_h\in V^h_r.
\]
Hence, for any $\psi_h\in V^h_{r+1}$, there holds
\begin{align*}
\|\partial_{x_i} v-\partial_{h,x_i}^* v\|_{L^2(\cT_h)}^2
&=\bigl( \partial_{x_i} v-\partial_{h,x_i}^* v,\partial_{x_i} v-\partial_{h,x_i}^* v \bigr)_{\cT_h} \\
&=\bigl( \partial_{x_i} v-\partial_{h,x_i}^* v,\partial_{x_i} v \bigr)_{\cT_h} \\
&=\bigl( \partial_{x_i} v-\partial_{h,x_i}^* v,\partial_{x_i} v-\partial_{x_i} \psi_h \bigr)_{\cT_h} \\
&\leq  \|\partial_{x_i} v-\partial_{x_i}^* v\|_{L^2(\cT_h)}
\|\partial_{x_i} v-\partial_{x_i} \psi_h\|_{L^2(\cT_h)}.
\end{align*}
The identity \eqref{e4.2} follows from the above inequality because 
$\partial_{x_i} V^h_{r+1} \subset V^h_r$. The second inequality 
\eqref{e4.2a} is obtained by similar arguments.
\end{proof}

It is also possible to derive some error estimates for 
$\mathcal{D}_{x_i} v-\partial_{h,x_i}^* v$ with general $v\in \cV_h$ 
in some weaker (than $L^2$) norm. However, since the derivation is
lengthy and technical,  we leave it 
to the interested reader to exploit.

\subsection{Product rule and chain rule}\label{sec-4.3}
The product rule and the chain rule are two very basic properties of
classical derivatives and weak derivatives 
(cf. \cite[Chapter 7]{Gilbarg_Trudinger01}). The goal of this subsection 
is to establish both of the rules for our DG FE derivatives. 
As expected, these discrete rules take different forms from
their continuous counterparts.

First, we consider the case when functions are from the space
$\cV_h\cap C^0(\Omega)$. In this case, we have the following 
product rule and chain rule.

\begin{theorem}\label{product_rule1}
Let $F\in C^1(\mathbf{R}), F'\in L^\infty(\mathbf{R})$. 
For $u,v\in \cV_h\cap C^0(\Omega)$, there holds, for $i=1,2,\cdots, d$,
\begin{align}\label{e4.12a}
\p_{h,x_i} (uv) &= \mathcal{P}^h_r \bigl(u\p_{x_i} v
+ v \p_{x_i} u\bigr),\\
\p_{h,x_i} F(u) &= \mathcal{P}^h_r \bigl( F'(u) \p_{x_i} u \bigr).
\label{e4.12b}
\end{align}
\end{theorem}

\begin{proof} 
Notice that $uv\in \cV_h\cap C^0(\Omega)$ for any $u,v\in 
\cV_h\cap C^0(\Omega)$.  By Proposition \ref{prop0a} and 
the product rule for weak derivatives 
(cf. \cite[Chapter 7]{Gilbarg_Trudinger01}) we get 
\begin{align*}
\bigl( \p_{h,x_i} (uv), \varphi_h \bigr)_{\cT_h} 
= \bigl( \p_{x_i} (uv), \varphi_h \bigr)_{\cT_h} 
= \bigl( u \p_{x_i} v +v \p_{x_i} u, \varphi_h \bigr)_{\cT_h}  
\qquad\forall \varphi_h\in V^h_r.
\end{align*}
Hence, \eqref{e4.12a} holds by the definition of $ \mathcal{P}^h_r$.  

Similarly, since $F(u)\in \cV_h\cap C^0(\Omega)$ 
for $u\in \cV_h\cap C^0(\Omega)$ we have 
\begin{align*}
\bigl( \p_{h,x_i} F(u), \varphi_h \bigr)_{\cT_h} 
= \bigl( \p_{x_i} F(u), \varphi_h \bigr)_{\cT_h} 
= \bigl( F'(u) \p_{x_i} u, \varphi_h \bigr)_{\cT_h}
\qquad\forall \varphi_h\in V^h_r,
\end{align*}
which then infers \eqref{e4.12a}. The proof is complete.
\end{proof}

An immediate consequence of Theorem \ref{product_rule1}
and Corollary \ref{cor4.1} is the following corollary.

\begin{corollary}
Let $F\in C^1(\mathbf{R}), F'\in L^\infty(\mathbf{R})$.
For $u,v\in V^h_\ell\cap C^0(\Omega)$
with $0\leq \ell-1\leq r$, there holds, for $i=1,2,\cdots, d$,
\begin{align}\label{e4.12c}
\p_{h,x_i} (uv) &= \mathcal{P}^h_r \bigl(u\p_{h,x_i} v
+ v \p_{h,x_i} u\bigr),\\
\p_{h,x_i} F(u) &= \mathcal{P}^h_r \bigl( F'(u) \p_{h,x_i} u \bigr).
\label{e4.12d}
\end{align}
\end{corollary}

Next, we consider the general case when functions are from the space $\cV_h$.
As expected, the discrete product and chain rules appear in more 
complicated forms.  

\begin{theorem}\label{product_rule2}
Let $F\in C^1(\mathbf{R})$ and $F'\in L^\infty(\mathbf{R})$. 
For $u,v\in \cV_h$, there holds for $i=1,2,\cdots, d$,
\begin{align}\label{e4.12e}
\p_{h,x_i} (uv) &= \mathcal{P}^h_r \bigl(\,\overline{u}\,\mathcal{D}_{x_i} v
+ \overline{v}\, \mathcal{D}_{x_i} u\bigr),\\
\p_{h,x_i} F(u) &= \mathcal{P}^h_r \bigl(\, \overline{F'(u)}\,\mathcal{D}_{x_i} u \bigr),
\label{e4.12f}
\end{align}
where $\mathcal{D}_{x_i} v$ is given by \eqref{e4.1b} and $\overline{v}$ 
represents a modification of $v$ which assumes the same value in each 
element $K\in \cT_h$ and takes the average value of $v$ on each interior 
edge $e\in \cE^I$. Thus,
\begin{align}\label{e4.12f0}
\overline{u}\,\mathcal{D}_{x_i} v= \sum _{K\in \cT_h} u\,\p_{x_i} v\, \chi_K 
-\sum_{e\in \cE^I} n_e^{(i)} \delta(e, \{u\}[v], x) 
\qquad \mbox{for a.e. } x\in \Omega.
\end{align}

\end{theorem}

\begin{proof}
Without loss of the generality, we assume $F\in C^\infty (\mathbf{R})$.
For $u,v\in \cV_h$, let $\{v_j\}, \{u_j\}\subset C^\infty_0(\Omega)$ 
be defined in Proposition \ref{prop0c}. Then, by Theorem \ref{product_rule1},
we have
\begin{align}\label{e4.12g}
\p_{h,x_i} (u_j v_j) &= \mathcal{P}^h_r \bigl(u_j\p_{x_i} v_j
+ v_j \p_{x_i} u_j \bigr),\\
\p_{h,x_i} F(u_j) &= \mathcal{P}^h_r \bigl( F'(u_j) \p_{x_i} u_j \bigr). 
\label{e4.12h} 
\end{align}

It follows from \eqref{e4.1g} and \eqref{e4.1e} that for any $\varphi_h\in 
V^h_r$,
\begin{align}\label{e4.12i}
\lim_{j\to\infty} \bigl(\p_{h,x_i} (u_j v_j), \varphi_h\bigr)_{\cT_h}
&=\bigl( \mathcal{P}^h_r \mathcal{D}_{x_i}(uv), \varphi_h\bigr)_{\cT_h}\\
&=\bigl\langle \mathcal{D}_{x_i}(uv), \varphi_h \bigr\rangle 
=\bigl(\p_{h,x_i} (u v), \varphi_h\bigr)_{\cT_h}, \nonumber \\
\lim_{j\to\infty} \bigl( \p_{h,x_i} F(u_j), \varphi_h\bigr)_{\cT_h}
&=\bigl( \mathcal{P}^h_r \mathcal{D}_{x_i} F(u), \varphi_h\bigr)_{\cT_h}
\label{e4.12j} \\
&=\bigl\langle \mathcal{D}_{x_i} F(u), \varphi_h \bigr\rangle
=\bigl(\p_{h,x_i} F(u), \varphi_h\bigr)_{\cT_h}. \nonumber 
\end{align}
On the other hand, we have 
\begin{align}\label{e4.12k}
\lim_{j\to\infty} \bigl( \mathcal{P}^h_r \bigl(u_j\p_{x_i} v_j 
+ v_j\p_{x_i} u_j\bigr), \varphi_h\bigr)_{\cT_h}
&= \lim_{j\to\infty} \bigl( u_j\p_{x_i} v_j + v_j\p_{x_i} u_j, 
\varphi_h\bigr)_{\cT_h} \\
&= \bigl\langle \mathcal{D}_{x_i} v, u\varphi_h  \bigr\rangle 
 + \bigl\langle \mathcal{D}_{x_i} u, v\varphi_h  \bigr\rangle \nonumber\\
&=\bigl\langle\, \overline{u}\,\mathcal{D}_{x_i} v
   +\overline{v}\, \mathcal{D}_{x_i} u, \varphi_h  \bigr\rangle, 
\nonumber \\
\lim_{j\to\infty} \bigl( \mathcal{P}^h_r \bigl( F'(u_j) \p_{x_i} u_j \bigr),
\varphi_h\bigr)_{\cT_h}
&=\lim_{j\to\infty} \bigl( F'(u_j) \p_{x_i} u_j, \varphi_h\bigr)_{\cT_h}
\label{e4.12l} \\
&=\bigl\langle \mathcal{D}_{x_i} u,F'(u) \varphi_h \bigr\rangle \nonumber\\
&=\bigl\langle\, \overline{F'(u)}\,\mathcal{D}_{x_i} u,\varphi_h \bigr\rangle. 
\nonumber 
\end{align}

Then, \eqref{e4.12e} follows from combining \eqref{e4.12g}, \eqref{e4.12i}
and \eqref{e4.12k}, and \eqref{e4.12f} follows from combining 
\eqref{e4.12h}, \eqref{e4.12j} and \eqref{e4.12l}. The proof is complete.
\end{proof}

\begin{remark}
The $C^1$ smoothness assumption on $F$ in Theorem {\rm \ref{product_rule2}}
may be weakened to $F\in C^{0,1}(\mathbf{R})$ by following the 
techniques used in \cite{Ambrosio_DalMaso90}. We leave the generalization 
to the interested reader to exploit.
\end{remark}

\subsection{Integration by parts formula and discrete divergence theorems}
\label{sec-4.4}

In this subsection, we derive integration by parts formulas for the 
discrete differential operators that resemble the standard integration by 
parts formula in the continuous setting. These results will play a 
crucial role in developing DG methods for a variety of PDE problems as well as 
in relating these methods to other existing DG methods in the literature.

\begin{theorem}
Suppose that $\gamma^+_i = - \gamma^-_i$, that is, $\gamma^+_{i,e} = -\gamma_{i,e}^-$ 
for all $e\in \mce^I$.  
Then the integration by parts formulas
\begin{align} \label{IBPFormula1}
\big( \p_{h,x_i}^\pm v_h, \varphi_h\big)_{\mct} 
&= -\big(v_h, \p_{h,x_i}^\mp \varphi_h\big)_{\mct}
+\bl v_h, \varphi_h n^{(i)}\br_{\mce^B}
\intertext{and}
\label{IBPFormula1b}
\big( \p_{h,x_i} v_h, \varphi_h\big)_{\mct}
& = -\big(v_h, \p_{h,x_i} \varphi_h \big)_{\mct}
+\bl v_h, \varphi_h n^{(i)}\br_{\mce^B}
\end{align}
hold for all $v_h,\varphi_h\in V_r^h$.  
\end{theorem}

\begin{proof}
It suffices to show \eqref{IBPFormula1}, since \eqref{IBPFormula1b}
trivially follows from \eqref{IBPFormula1}.

By \eqref{e2.5} and integration by parts, we obtain
\begin{align*}
\big(\p_{h,x_i}^\pm v, \varphi_h\big)_{\mct}
= \big( \p_{x_i} v, \varphi_h\big)_{\mct} 
&+\Bl  \big(\cQ^\pm_i (v)  - \{v\}\big),[\varphi_h]n^{(i)}\Br_{\mce} \\
&-\bl  [v],\{\varphi_h\}n^{(i)}-\gamma^\pm_i [\varphi_h]\br_{\mce^I}.
\end{align*}
Using the identity \eqref{Qidentity}, we have
\begin{align}\label{IBPLine}
\big( \p_{h,x_i}^\pm v, \varphi_h\big)_{\mct}
& = \big(\p_{x_i} v, \varphi_h\big)_{\mct} 
\pm\frac12 \bl    [v], [\varphi_h] {\rm sgn}(n^{(i)}) n^{(i)}\br_{\mce^I} \\
&\hskip 1in
-\bl  [v],\{\varphi_h\}n^{(i)}-\gamma^\pm_i  [\varphi_h]\br_{\mce^I} \nonumber \\
& = \nonumber \big( \p_{x_i} v, \varphi_h\big)_{\mct}
+\Bl [v],(\gamma^\pm_i \pm \frac12 |n^{(i)}|) [\varphi_h] -\{\varphi_h\}n^{(i)}\Br_{\mce^I}.
\end{align}

Now let $v = v_h\in V^h_r$.  Then by \eqref{e2.5}, \eqref{Qidentity} and 
rearranging terms, we have
\begin{align}\nonumber
( \p_{x_i} v_h, \varphi_h\big)_{\mct} 
&= -\big( \p_{h,x_i}^\mp \varphi_h,v_h\big)_{\mct}
+\bl \cQ^\mp_i (\varphi_h),[v_h]n^{(i)}\br_{\mce}
+\bl  \gamma^\mp_i [\varphi_h],[v_h]\br_{\mce^I}\\
&\label{IBPLine2}
 = -\big( \p_{h,x_i}^\mp \varphi_h, v_h\big)_{\mct}
+\bl \{\varphi_h\},[v_h]n^{(i)}\br_{\mce^I}\\
&\nonumber\hskip 0.5in
+\bl  (\gamma^\mp_i \mp \frac12 |n^{(i)}|) [\varphi_h],[v_h]\br_{\mce^I}
+ \bl \varphi_h, v_h n^{(i)}\br_{\mce^B}.
\end{align}
Using the identity \eqref{IBPLine2} in equation \eqref{IBPLine}, we obtain
\begin{align}\label{GeneralIBP}
\big( \p_{h,x_i}^\pm v_h, \varphi_h\big)_{\mct}
= -\big( v_h,\p_{h,x_i}^\mp \varphi_h\big)_{\mct}
& +\bl  v_h, \varphi_h n^{(i)}\br_{\mce^B} \\
&+\bl  (\gamma ^\pm_i+\gamma^\mp_i)  [\varphi_h],[v_h]\br_{\mce^I}. \nonumber
\end{align}
The identity \eqref{IBPFormula1} then easily follows.
\end{proof}

\begin{theorem}\label{FirstCorollary}\
Suppose that $\gamma^+_i = - \gamma^-_i$ for all $i=1,2,\ldots d$.  
Then the integration by parts formulas
\begin{align}\label{DivFormula1}
\big( \Div^\pm_h \bv_h,  \varphi_h\big)_{\mct} 
= -\big( \bv_h, \nab_h^\mp \varphi_h\big)_{\mct}
+\bl  \bv_h \cdot n, \varphi_h\br_{\mce^B}
\intertext{and}
\label{DivFormula1b}
\big(\Div_h \bv_h,  \varphi_h\big)_{\mct} 
= -\big( \bv_h, \nab_h \varphi_h\big)_{\mct}
 +\bl \bv_h \cdot n, \varphi_h\br_{\mce^B}
\end{align}
hold for all $\bv_h\in \bV_r^h$ and $\varphi_h\in V_r^h$.  
\end{theorem}
 
\begin{proof}
If $\gamma^+_e = - \gamma^-_e$, then by \eqref{e2.11},
 \eqref{IBPFormula1} and \eqref{e2.8}, we have
\begin{align*}
\big(\Div_h^\pm \bv_h,\varphi_h\big)_{\mct}
& = \sum_{i=1}^d \big(\p_{h,x_i}^\pm v_{h,i}, \varphi_h\big)_{\mct} \\
&=\sum_{i=1}^d \Big( -\big( \p_{h,x_i}^\mp \varphi_h, v_{h,i}\big)_{\mct}
 + \bl v_{h,i}n^{(i)}, \varphi_h\br_{\mce^B}\Big)\\
& = -\big(\bv_h, \nab^\mp_h \varphi_h\big)_{\mct} +\bl   \bv_h \cdot n, 
\varphi_h\br_{\mce^B}.
\end{align*}
Formula \eqref{DivFormula1b} is obtained similarly.
\end{proof}

\begin{theorem}\label{SecondCorollary}
The formal adjoint of the operator $\Div_h^\pm $ (resp., $\Div_h$)
is $-\nab_{h,0}^{\mp}$ (resp., $-\nab_{h,0}$) with respect
to the inner product $(\cdot,\cdot)_{\mct}$ provided $\gamma_i^+=-\gamma_i^-$
for all $i=1,2,\ldots,d$; that is,
\begin{align}
(\Div_h^\pm \bv_h,\varphi_h)_{\mct} &= -(\bv_h,\nab_{h,0}^\mp \varphi_h)_{\mct}, \\
(\Div_h \bv_h,\varphi_h)_{\mct} &= -(\bv_h,\nab_{h,0} \varphi_h)_{\mct}
\end{align}
for all $\bv_h\in \bV_r^h,\ \varphi_h\in V_r^h.$
In addition, if $\gamma_i^+ = -\gamma_i^-$, then 
the formal adjoint of the operator $\Div_{h,{\bm 0}}^\pm$ (reps., $\Div_{h,{\bm 0}}$)
is $-\nab_{h}^{\mp}$ (resp., $-\nab_h$). 
\end{theorem}
\begin{proof}
This result immediately follows from Theorem \ref{FirstCorollary}
and the identities
\begin{align*}
\big(\bv,\nab_{h,0}^\pm \varphi_h\big)_{\mct} 
&= \big(\bv, \nab_h^\pm \varphi_h\big)_{\mct} 
-\bl \bv\cdot n,\varphi_h\br_{\mce^B},\\
\big(\Div_{h,0}^\pm \bv,\varphi_h\big)_{\mct} 
&= \big(\Div_h^\pm \bv,\varphi_h\big)_{\mct}
-\bl \bv\cdot n,\varphi_h\br_{\mce^B}
\end{align*}
for all $\bv\in \bV_h$ and $\varphi_h\in V_r^h$.
\end{proof}

\begin{theorem}\label{ThirdCorollary}
Suppose that $\gamma^+ = - \gamma^-$.  We then have
\begin{subequations}
\begin{align}\label{LapFormula1}
-\big( \Delta_h^{\pm +} v, \varphi_h\big)_{\mct} 
&= \big( \nab^+_h v, \nab^\mp_h \varphi_h\big)_{\mct}
-\bl  \nab_h^+ v\cdot n, \varphi_h\br_{\mce^B},\\
-\big(\Delta_h^{\pm -} v,\varphi_h\big)_{\mct}
& = \big( \nab^-_h v, \nab^\mp_h \varphi_h\big)_{\mct}
- \bl \nab_h^- v\cdot n, \varphi_h\br_{\mce^B}\\
\intertext{and}
-\big( \Delta_h v,\varphi_h\big)_{\mct}
& = \big( \nab_h v, \nab_h \varphi_h\big)_{\mct}
-\bl  \nab_h v\cdot n, \varphi_h\br_{\mce^B}
\end{align}
\end{subequations}
for all $v\in \cV_h$ and $\varphi_h\in V_r^h.$
In addition, under the same hypotheses on $\gamma_i^\pm$, we have
\begin{subequations}
\begin{align}
-\big( \Delta_{h,g}^{\pm +} v, \varphi_h\big)_{\mct} 
&= \big( \nab^+_{h,g} v, \nab^\mp_{h,0} \varphi_h\big)_{\mct},\\
-\big(\Delta_{h,g}^{\pm -} v,\varphi_h\big)_{\mct}
& = \big( \nab^-_{h,g} v, \nab^\mp_{h,0} \varphi_h\big)_{\mct}
\intertext{and}
-\big( \Delta_{h,g} v,\varphi_h\big)_{\mct}
& = \big( \nab_{h,g} v, \nab_{h,0} \varphi_h\big)_{\mct}.
\end{align}
\end{subequations}

\end{theorem}

\begin{proof}
These formulas follow from Theorems \ref{FirstCorollary}--\ref{SecondCorollary}
with $\bv_h = \nab_h^\pm v\in \bV_r^h$ and $\bv_h = \nab_{h,g}^\pm v\in \bV_r^h$ 
(cf. Remark \ref{rem2}c).
\end{proof}

\subsection{Relationships with finite difference operators on Cartesian grids}
\label{sec-4.5}
We now show that when $\gamma_e^\pm =0$ the operators 
$\partial_{h,x_i}^\pm$ and $\partial_{h,x_i}$ are natural extensions 
(on general grids) of the backward/forward and central difference 
operators defined on Cartesian grids. 

Suppose $\cT_h$ is a rectangular mesh over $\Omega$ which is aligned
with the underlying Cartesian coordinate system. Let $h_i$ denote the 
mesh size of $\cT_h$ in the direction of $x_i$.  
Notice that when $r=0$, the function $\partial_{h,x_i}^\pm v$ is
a piecewise constant function over the mesh $\cT_h$. Setting 
$\varphi=1$ in \eqref{e2.5a} we get
\begin{align}\label{e4.5}
\bigl( \partial_{h,x_i}^\pm v, 1 \bigr)_K 
= \langle \cQ_i^\pm(v) n_K^{(i)}, 1\rangle_{\partial K}.
\end{align}

We only consider the two dimensional case. Then $K$ has four edges with 
$n_K=(-1,0)^t,\ (1,0)^t,\ (0,1)^t,\ (0,-1)^t$. Let $v$ be a grid function over
$\cT_h$ and $v_{ij}$ denote the value of $v$ on the $(i,j)$ cell/element.
By \eqref{e4.5} we obtain
\begin{align}\label{e4.6}
&\partial_{h,x_1}^+ v_{ij} = \frac{ v_{i+1 j}-v_{ij} }{h_1},\qquad
\partial_{h,x_1}^- v_{ij} = \frac{ v_{ij}-v_{i-1 j} }{h_1},\\
&\partial_{h,x_2}^+ v_{ij} = \frac{ v_{ij+1}-v_{ij} }{h_2},\qquad
\partial_{h,x_2}^- v_{ij} = \frac{ v_{ij}-v_{ij-1} }{h_2}.  \label{e4.7}
\end{align}
Consequently, we have
\begin{align}\label{e4.8}
&\partial_{h,x_1} v_{ij} = \frac{ v_{i+1 j}-v_{i-1j} }{2h_1},  \qquad
\partial_{h,x_2} v_{ij} = \frac{ v_{ij+1}-v_{ij-1} }{2h_2}.
\end{align}
Hence,  $\partial_{h,x_i}^-$ and $\partial_{h,x_i}^+$ coincide respectively
with backward and forward difference operators in the direction $x_i$, 
while $\partial_{h,x_i}$ results in the central difference operator in the 
direction $x_i$.

Using the operators $\partial_{h,x_i}^\pm$ as building blocks, we can also 
build DG FE operators that are natural extensions of the standard finite 
difference approximations for second order partial derivatives.
From \eqref{e4.6} and \eqref{e4.7}, we have 
\begin{align}
& \partial_{h,x_1}^+ \partial_{h,x_1}^- v_{ij} = \partial_{h,x_1}^- \partial_{h,x_1}^+ v_{ij}
	= \frac{ v_{i-1 j} - 2 v_{i j} + v_{i+1 j}}{h_1^2}, 
	\label{e4.9} \\
& \partial_{h,x_1} \partial_{h,x_1} v_{ij} 
	= \frac{ v_{i-2 j} - 2 v_{i j} + v_{i+2 j}}{4 h_1^2}, 
	\label{e4.10} \\
&\frac{ \partial_{h,x_1}^+ \partial_{h,x_2}^- v_{ij} 
+ \partial_{h,x_1}^- \partial_{h,x_2}^+ v_{ij}}{2} \label{e4.11} \\ 
	& \hskip 0.6in = \frac{v_{i-1 j} + v_{i j-1} - v_{i-1 j+1} - 2 v_{i j} - v_{i+1 j-1} + v_{i j+1} + v_{i+1 j}}{2h_1 h_2}, \nonumber \\
& \partial_{h,x_1} \partial_{h,x_2}
	= \frac{\partial_{h,x_1}^+ \partial_{h,x_2}^+ v_{ij} 
+  \partial_{h,x_1}^+ \partial_{h,x_2}^- v_{ij} 
+ \partial_{h,x_1}^- \partial_{h,x_2}^+ v_{ij} 
+ \partial_{h,x_1}^- \partial_{h,x_2}^- v_{ij} }{4} \label{e4.12} \\
	& \hskip 0.62in  = \frac{ v_{i+1 j+1} - v_{i+1 j-1} - v_{i-1 j+1} 
+ v_{i-1 j-1}}{4h_1 h_2}. \nonumber
\end{align}
Thus, for $k \neq \ell$, the discrete differential operator 
$(\partial_{h,x_k}^+ \partial_{h,x_k}^- 
+ \partial_{h,x_k}^- \partial_{h,x_k}^+ ) / 2$ coincides with the 
standard second order 3-point central difference operator for nonmixed 
second order derivatives, $\partial_{h,x_k} \partial_{h,x_k}$ coincides 
with the standard second order 3-point central difference operator 
with mesh size $2h$, 
$(\partial_{h,x_k}^+ \partial_{h,x_\ell}^- 
+ \partial_{h,x_k}^- \partial_{h,x_\ell}^+ ) / 2$ coincides with the 
second order $7$-point central difference operator for mixed second 
order derivatives, and $\partial_{h,x_k} \partial_{h,x_\ell}$ 
coincides with the standard second order 
central difference operator for mixed second order derivatives.

\section{Implementation Aspects}\label{sec-5}

In this section we explain how the discrete partial derivatives are computed.
Denote by $\hat{K}:=\big\{x\in \mathbf{R}^d:\ x_i\ge 0,\ \sum_{i=1}^d x_i<1\big\}$ the reference simplex,
and let $\{\hat{\varphi}^{(j)}_h\}_{j=1}^{n_r}$ denote a basis of $\pol_r(\hat{K})$.
Here $n_r = \dim \pol_r = \binom{d+r}{r}$.  For example,
we could take the basis to be the space of monomials of degree  less than 
or equal to $r$: $\{x^\alpha:\ |\alpha|\le r\}$.
For $K\in \mct$, let $F_K:\hat{K}\to K$ be the affine mapping from $\hat{K}$ onto $K$,
and and let $\varphi^{(j,K)}_h:K\to \mathbf{R}$ be defined by $\varphi^{(j,K)}_h(x) = \hat{\varphi}_h^{(j)}(\hat{x})$, 
where $x = F_K(\hat{x})\in K$.  It is then easy to see that $\{\varphi^{(j,K)}_h\}_{j=1}^{n_r}$
is a basis of $\pol_r(K)$.  
We then define the mass matrix $\tilde{\bm M}_K\in \mathbf{R}^{n_r\times n_r}$
associated with $K$ as
\begin{align*}
\big(\tilde{\bm M}_K\big)_{\ell ,m } = \bigl(\varphi_h^{(\ell)},\varphi_h^{(m)} \bigr)_K\qquad \ell,m=1,2,\ldots,n_r.
\end{align*}
By a change of variables we easily find $\tilde{\bm M}_K = |\det(DF_K)| \tilde{\bm M} = d! |K| \tilde{\bm M}$, 
where $\tilde{\bm M}$ is the mass matrix associated with $\hat{K}$, $DF_K$ is the Jacobian
of the mapping $F_K$ and $|K|$ is the $d$-dimensional volume of the simplex $K$.

Next, given $v\in \cV_h$,
write $\p_{h,x_i}^\pm v\big|_K = \sum_{j=1}^{n_r} \alpha^{\pm (j)}_i \varphi^{(j,K)}_h\in \pol_r(K)$ with 
$\alpha^{\pm (j)}_i = \alpha^{\pm (j)}_{i,K}\in \mathbf{R}$
($i=1,2,\ldots,d,\ j=1,2,\ldots,n_r)$.
We then define the vector ${\bm b}^\pm_{i} = {\bm b}^\pm_{i,K}(v)\in \mathbf{R}^{n_r}$ by
\begin{align*}
{\bm b}_i^{\pm (j)} 
&= \bl \cQ_i^\pm(v)n_K^{(i)},\varphi^{(j,K)}_h\br_{\p K} - \big(v,\p_{x_i} \varphi_h^{(j,K)}\big)_K\\
&\hskip 0.5in 
+\sum_{e\subset \p K\backslash \p \Ome} \gamma_{i,e}^\pm \bl [v],[\varphi_h^{(j,K)}]\br_e, \qquad
j=1,2,\ldots,n_r.
\end{align*}
Then by \eqref{e2.5a}, the coefficients $\{\alpha_i^{(j)}\}_{j=1}^{n_r}$ are uniquely determined
by the linear equation $M_K {\bm \alpha^\pm_i} = {\bm b}^\pm _i$,
where ${\bm \alpha}^\pm_i = (\alpha^{\pm (1)}_i,\alpha^{\pm (2)}_i,\ldots ,\alpha^{\pm (n_r)}_i)^t$.
Equivalently, we have ${\bm \alpha}_i^\pm  = \frac{1}{d! |K|} \tilde{\bm M}^{-1} {\bm b}_i^\pm$,
where $\tilde{\bm M}^{-1}$ is the inverse matrix of the reference mass matrix which
can be computed offline.

Now if $v = v_h \in V_r^h$, then we may write $v_h\big|_K = \sum_{j=1}^{n_r} \beta^{(j)} \varphi_h^{(j,K)}$
for some constants $\beta^{(j)} = \beta^{(j,K)}\in \mathbf{R}$.  We then define the matrix
$\tilde{\bm Q}^\pm_i = \tilde{\bm Q}_{i,K}^\pm \in \mathbf{R}^{n_r\times n_r}$ by
\begin{align*}
\big(\tilde{\bm Q}_i^\pm\big)_{\ell,m} 
&= \bl \cQ^\pm_i (\varphi_h^{(m,K)})n_K^{(i)},\varphi_h^{(\ell,K)}\br_{\p K}
 -\big(\varphi_h^{(m,K)},\p_{x_i} \varphi_h^{(\ell,K)}\big)_K,\\
 &\qquad + \sum_{e\subset \p K\backslash \p \Ome} \gamma_{i,e}^\pm \bl [\varphi_h^{(\ell,K)}],[\varphi_h^{(m,K)}]\br_e
\end{align*}
for $\ell,m = 1,2,\ldots,n_r$.  Again, writing
$\p_{h,x_i}^\pm v\big|_K = \sum_{j=1}^{n_r} \alpha_i^{\pm (j)} \varphi_h^{(j,K)}$,
we find that the coefficients satisfy ${\bm \alpha}_i^\pm  
= \frac{1}{d! |K|} \tilde{\bm M}^{-1} \tilde{\bm Q}_i^\pm {\bm \beta}$,
where ${\bm \beta} = \big(\beta^{(1)},\beta^{(2)},\ldots,\beta^{(n_r)}\big)^t$.

\section{Applications}\label{sec-6}

In this section, we apply our DG FE differential calculus 
framework to constructing numerical methods for several different
types of PDEs. Essentially, we replace the (continuous)
differential operator by its discrete counterpart.  In addition,
we add a stability term which ensures that the discrete problem
is well-posed and also enforces the given boundary conditions
weakly within the variational formulation.
Throughout the section, we assume the penalty parameters $\gamma_i$
(which appear in the definition of the discrete differential operators)
are zero in the discussion below. 

\subsection{Second order linear elliptic PDEs}\label{sec-5-1}
\subsubsection{\bf The Poisson equation}\label{sec-5-1-1}

As our first example, we consider the simplest linear second order 
PDE, the Poisson equation with Dirichlet boundary conditions:
\begin{subequations}
\label{Poisson}
\begin{alignat}{2}
\label{PoissonA}
-\Del u & = f\qquad &&\text{in }\Ome,\\
\label{PoissonB}
u & = g\qquad &&\text{on }\p\Ome,
\end{alignat}
\end{subequations}
where $f\in L^2(\Ome)$ and $g\in L^2(\p\Ome)$ are two given functions.
We then consider the following discrete version of \eqref{Poisson}: Find $u_h\in V_h^r$ such that
\begin{align}\label{DPoisson}
-\Del_{h,g} u_h +j_{h,g}(u_h) = \cP_r^hf .
\end{align}
Here, $j_{h,g}(\cdot):\cV_h\to V^h_r$ is the unique operator satisfying
\begin{align} \label{jhDef}
(j_{h,g}(v),\varphi_h)_{\mct} = \bl \eta_1 [v],[\varphi_h]\br_{\mce^I} + \bl \eta_1 (v-g),\varphi_h\br_{\mce^B},
\end{align}
and $\eta_1$ is a penalty parameter 
that is piecewise constant with respect to the set of edges.

Problem \eqref{DPoisson} has several interpretations.
On the one hand, by the definition of the discrete Laplace operator,
the problem is the twofold saddle point problem
\begin{align*}
-{\rm tr}(\tilde{\bm r}_h) +j_h (u_h) &= \cP_r^h f,\\
\tilde{\bm r}_h = D_h \bq_h,\quad \bq_h &= \nab_{h,g} u_h,
\end{align*}
with $\bq_h\in \bV_r^h$ and $\tilde{\bm r}_h \in \widetilde{\bm V}_r^h$.
Namely, problem \eqref{DPoisson} is equivalent 
to finding $u_h\in V_h$, $\bq_h\in \bV^r_h$ and $\tilde{\bm r}_h\in \widetilde{\bm V}_r^h$
such that 
\begin{subequations}
\label{dualmixed}
\begin{alignat}{2}
\label{dualmixed1}
& (\bq_h,\btau_h)_{\mct} 
 = \bl \{u_h\},[\btau_h]\cdot n\br_{\mce^I} - (u_h,\Div \btau_h)_{\mct}
 +\bl g,\btau_h\cdot n\br_{\mce^B}\quad
&&\forall \btau_h\in \bV_r^h,\\
\label{dualmixed2}
& \big(\us{{\bm r}}_h,{\us {\bm \mu}}_h\big)_{\mct} 
 = \bl \{\bq_h\},[\tilde{\bm \mu}_h]\br_{\mce} - (\bq_h,\Div \tilde{\bm \mu}_h)_{\mct}
\quad &&\forall \tilde{\bm \mu}_h\in \widetilde{\bm V}_h^r,\\
\label{dualmixed3}
&-\big({\rm tr}(\tilde{\bm r}_h),\varphi_h\big)_{\mct} 
+\big(j_{h,g}(u_h),\varphi_h\big)_{\mct}
 = (f,v)_{\mct}\quad &&\forall \varphi_h\in V_r^h.
\end{alignat}
\end{subequations}
On the other hand,  by Theorem \ref{FirstCorollary},
we can write problem \eqref{DPoisson} in its primal form:
find $u_h\in V^h_r$ satisfying
\begin{align}\label{DPoissonE}
(\nab_{h,g} u_h,\nab_h \varphi_h)_{\mct}
-\bl \nab_{h,g} u_h\cdot n,\varphi_h\br_{\mce^B} 
+\big(j_{h,g}(u_h),\varphi_h\big)_{\mct}
& = (f,\varphi_h)_{\mct}
\end{align}
for all $\varphi\in V_r^h$. Alternatively, by Theorem \ref{SecondCorollary},
we may write
\begin{align*}
(\nab_{h,g} u_h,\nab_{h,0} \varphi_h)_{\mct}
+\big(j_{h,g}(u_h),\varphi_h\big)_{\mct}
& = (f,\varphi_h)_{\mct}\qquad \forall \varphi_h\in V_r^h.
\end{align*}
Finally,  in the case $g=0$, problem \eqref{DPoissonE} can
be  viewed as finding a minimizer of the functional 
\begin{align}\label{FunctionalJ}
v_h \to  \frac12 \int_\Ome |\nab_{h,0} v_h|^2\, dx 
+\sum_{e\in \mce} \frac12 \int_e \eta_1 \big|[v_h]\big|^2\, ds
-\int_\Ome fv_h\, dx
\end{align}
over all $v_h\in V^h_r$.

We now discuss the well-posedness of problem \eqref{DPoisson} as well as relate
the discretization to other DG schemes. 
Let $\bq_h = \nab_{h,g} u_h$, that is, $\bq_h\in \bV_r^h$ satisfies \eqref{dualmixed1}.\
Then by \eqref{DPoissonE}, we have
\begin{align}\label{DPoissonEwithq}
(\bq_h,\nab_h \varphi_h)_{\mct}
 -\bl \bq_h \cdot n,\varphi_h\br_{\mce^B}  +\big(j_{h,g}(u_h),\varphi_h\big)_{\mct}
= (f,\varphi_h)_{\mct}
\end{align}
for all $\varphi_h\in V^h_r$.  
By the definition of the discrete gradient and integration by parts, we have
\begin{align*}
(\bq_h,\nab_h \varphi_h)_{\mct}
& = \bl [\bq_h]\cdot n,\{\varphi_h\} \br_{\mce} -(\Div \bq_h,\varphi_h,)_{\mct} \\
& = (\bq_h,\nab \varphi_h)_{\mct} -\bl \{\bq_h\}\cdot n,[\varphi_h]\br_{\mce^I}.
\end{align*}
Using this identity in \eqref{DPoissonEwithq}, we find
\begin{align}\label{LDG2}
(\bq_h,\nab \varphi_h)_{\mct} -\bl \{\bq_h\}\cdot n,[\varphi_h]\br_{\mce} 
+\big(j_{h,g}(u_h),\varphi_h\big)_{\mct}
& = (f,\varphi_h)_{\mct}\quad
\forall \varphi\in V_r^h.
\end{align}

In summary, problem \eqref{DPoisson} is equivalent to the 
mixed formulation \eqref{dualmixed1}, \eqref{LDG2}.
This formulation  is nothing more than the local discontinuous Galerkin 
(LDG) method  \cite{Cockburn98,abcm01}.  
We then have (see, e.g., \cite{Castilo00})
\begin{theorem}\label{LDGLemma}
Let $r\ge 1$, $\eta_1>0$ and $\gamma_i = 0\ (i=1,2,\ldots,d)$.  Then there exists a 
unique $u_h\in V^h_r$ satisfying \eqref{DPoisson}.  Moreover,
if $\eta_1 = \mathcal{O}(h^{-1})$ and if $u\in H^{r+2}(\Ome)$,  there holds
\[
\|u-u_h\|_{L^2(\Ome)}+h \|\nab u-\nab_{h,g} u_h\|_{L^2(\Ome)}\le C h^{r+1}\|u\|_{H^{r+2}(\Ome)}.
\]
\end{theorem}

\begin{remark}
A similar methodology can be used to construct
DG schemes for the Neumann problem
\begin{align}
\label{NMPoisson}
-\Del u & = f\quad \text{in }\Ome,\qquad
\frac{\p u}{\p n}  = q\quad \text{on }\p\Ome.
\end{align}
In this case, the DG method is to find $u_h\in V^h_r$ satisfying
\begin{align*}
-\Div_{h,\bq} \nab_h u_h +\tilde{j}_h(u_h) = \cP_r^h f,
\end{align*}
where $\tilde{j}_h(u_h):\mathcal{V}^h\to V^h_r$ is 
the operator satisfying 
$\big(\tilde{j}_h(v),\varphi_h\big)_{\mct} = \bl\eta_1  [v],[\varphi_h]\br_{\mce^I}$ 
for all $\varphi_h\in V_r^h$, and $\bq = q n$.  We again recover
the LDG method for the Neumann problem \eqref{NMPoisson} (cf.\,\cite{Castilo00}).
\end{remark}
\subsubsection{\bf The DWDG method for the Poisson problem} \label{dwdg}
In this subsection we formulate a new DG method for the Poisson 
problem \eqref{Poisson} that inherits better stability than the LDG method 
described above.  The new scheme, called the symmetric 
dual-wind discontinuous Galerkin (DWDG) method, 
is simply given by
\begin{align}\label{AltPoisson}
- \frac{\Del_{h,g}^{-+}u_h+\Del_{h,g}^{+-} u_h}{2} + j_{h,g}(u_h)& = \cP_r^h f, 
\end{align}
where $j_{h,g}(u_h)$ is defined by \eqref{jhDef}.
By Theorem \ref{ThirdCorollary}, 
 problem \eqref{AltPoisson} is equivalent to finding $u_h\in V^h_r$ such that
\begin{align}\label{AltPoissonPrimal}
&\frac12\Big( (\nab_{h,g}^+ u_h,\nab_{h,0}^+ v_h)_{\mct}+(\nab_{h,g}^- u_h,\nab_{h,0}^- v_h)_{\mct}\Big)
+ j_{h,g}(u_h) = (f,v_h)_{\mct}
\end{align}
for all $v_h\in V^h_r$.
Equivalently, in the case $g=0$, problem \eqref{AltPoisson} asks to
find the unique minimizer
of the functional 
\begin{align*}
v_h \to \frac14 \int_\Ome \Big(|\nab_{h,0}^+ v_h|^2+|\nab_{h,0}^- v_h|^2\Big)\, dx 
+\sum_{e\in \mce} \frac12 \int_e \eta_1 \big|[v_h]\big|^2\, ds
-\int_\Ome fv_h\, dx
\end{align*}
over all $v_h\in V^h_r$ (compare to \eqref{FunctionalJ}).
A complete convergence analysis of the symmetric DWDG method for the Poisson problem 
is presented in \cite{LewisNeilan12}.  Here, we summarize the main results, namely, 
well-posedness and optimal rates of convergence.

\begin{theorem}[\cite{LewisNeilan12}]\label{UniquenessTheorem}
Set $\gamma _{\min} :=\min_{e\in \mce} h_e^{-1}\eta_1(e)$.
Suppose that there exists at least one simplex $K \in\mct$
with exactly one boundary edge/face.
Then there exists a unique solution to \eqref{AltPoisson} provided $\gamma_{\min}\ge 0$.
Furthermore, if the triangulation is quasi-uniform, and if each simplex in the triangulation has at most 
one boundary face/edge, then there exists a constant $C_* > 0$ independent of $h$ and $\eta_1$ 
such that
problem \eqref{AltPoisson} has a unique solution provided $\gamma_{\min} > - C_*$.
\end{theorem}

\begin{remark}
We emphasize that problem \eqref{AltPoisson} is well-posed
without added penalty terms.
As far as we are aware, this is the first symmetric 
DG method that has this property in any dimension (cf.~\cite{OdenEtal,LarsonNiklasson04,BurmanEtal07}).
\end{remark}

\begin{theorem}[\cite{LewisNeilan12}]\label{L2EstimateTheorem}
Let $u_h$ be the solution to \eqref{AltPoisson}, 
$u \in H^{s+1}(\Omega)$ be the solution to \eqref{Poisson} and $\gamma_{\max} = \max_{e\in \mce} h_e^{-1}\eta_1(e)$.
Then $u_h$ satisfies the following estimate provided $\gamma_{\min}>0$:
\begin{align}\label{L2EstimateLine1}
\|u-u_h\|_{L^2(\Ome)}\le C h^{s+1}\Big(\sqrt{\gamma_{\max}}+\frac{1}{\sqrt{\gamma_{\min}}}\Big)^2|u|_{H^{s+1}(\Ome)}\qquad
(1\le s\le r), 
\end{align}
and if the triangulation is quasi-uniform and $\gamma_{\min}>-C_*$, then there holds
\begin{align}\label{L2EstimateLine2}
\|u-u_h\|_{L^2(\Ome)}\le C h^{s+1}\Big(\sqrt{|\gamma_{\min}|} +\frac{1}{\sqrt{C_*+\gamma_{\min}}}\Big)^2 |u|_{H^{s+1}(\Ome)}, 
\end{align}
where $C$ denotes a generic positive constant independent of $h$, and $C_*$ is the positive constant from 
Theorem~{\rm\ref{UniquenessTheorem}}.
\end{theorem}

\begin{remark}
In light of \eqref{e4.9} and \eqref{e4.10}, we can see that
when approximating with piecewise constant basis functions 
on Cartesian grids, 
the DWDG method coincides with the standard 
finite difference method for Poisson's equation while the LDG method coincides with a staggered 
finite difference method for Poisson's equation that uses coarser second derivative approximations.
\end{remark}

\subsection{Fourth order linear PDEs}
In this subsection we show how to use the discrete differential calculus
to develop DG methods for fourth order linear PDEs.  For simplicity, 
we focus our derivation to the model biharmonic problem with
clamped boundary conditions:
\begin{subequations}
\label{biharmonic}
\begin{alignat}{2}
\Del^2 u & = f\qquad &&\text{in }\Ome,\\
u &= g\qquad &&\text{on }\p\Ome,\\
\frac{\p u}{\p n } & = q &&\text{on }\p\Ome,
\end{alignat}
\end{subequations}
where we assume that $f\in L^2(\Ome)$ and $g,q\in L^2(\p\Ome)$.
The biharmonic problem with other boundary conditions (e.g., simply supported
or Cahn-Hilliard type) are briefly discussed at the end of
the section.

To develop DG methods using the discrete differential machinery,
we first need to define some additional discrete operators corresponding
to the biharmonic operator.  Given the two functions $g$ and $q$
defined on the boundary, we set
\begin{align}
\label{DiscreteBiharmonicDef}
\Del_{h,g,\bq}:&=\Div_{h,\bq}\nab_{h,g}\quad \text{with }\bq = q n,
\intertext{and}
\Del^2_{h,g,\bq}:&=\Del_h \big(\Del_{h,g,\bq})= \Div_h \nab_h \Div_{h,\bq} \nab_{h,g}.
\end{align} 
In addition, we define the operator $r_{h,q}(\cdot): W^{2,1}(\mct)\cap C^1(\mct)\to V_r^h$ by
\begin{align*}
\big(r_{h,q}(v),\varphi_h\big)_{\mct} = \bl \eta_2 [\p v/\p n],[\p \varphi_h/\p n]\br_{\mce^I} 
+ \bl \eta_2 (\p v/\p n-q),\p \varphi_h/\p n\br_{\mce^B},
\end{align*}
which we will use to enforce the Neumann boundary condition weakly in the DG formulation.

The DG method for \eqref{biharmonic} is then defined as seeking $u_h\in V^h_r$ such that
\begin{align}
\label{Dbiharmonic}
\Del^2_{h,g,\bq} u_h +j_h(u_h)+r_h(u_h) = \cP_r^h f.
\end{align}
Similar to the discussion in Section \ref{sec-5-1}, we may write
the DG method in various mixed forms (with up to six unknown variables).
Instead, we focus mainly on the primal formulation.

By \eqref{DiscreteBiharmonicDef} and Theorem \ref{SecondCorollary}, we have
\begin{align*}
\big(\Del^2_{h,g,\bq} u_h,\varphi_h\big)_{\mct}  
& = -\big(\nab_h \Del_{h,g,\bq} u_h,\nab_{h,0} \varphi_h\big)_{\mct}
 = \big(\Del_{h,g,\bq} u_h,\Div_{h,0} \nab_{h,0} \varphi_h\big)_{\mct}\\
& = \big(\Del_{h,g,\bq} u_h,\Del_{h,0,0} \varphi_h\big)_{\mct}.
\end{align*}
Thus, we may write \eqref{Dbiharmonic} in its primal formulation as follows:  Find $u_h\in V^h_r$ such that
\begin{align}
\big(\Del_{h,g,\bq} u_h,\Del_{h,0,0} \varphi_h\big)_{\mct} 
&+ \big(j_{h,g}(u_h),\varphi_h\big)_{\mct}+\big(r_{h,\bq}(u_h),\varphi_h\big)_{\mct}\\
& \nonumber
= (f,\varphi_h\big)_{\mct}\qquad \forall \varphi_h\in V^h_r.
\end{align}

\begin{remark}
The DG method \eqref{Dbiharmonic} closely 
resembles the local continuous discontinuous Galerkin (LCDG)
method proposed in \cite{HuangEtal10} (also see \cite{WellsDung07}).
Here, the authors consider a mixed formulation
of the biharmonic problem with the Hessian of $u$ as an additional unknown.
The derivation of the LCDG method closely resembles the derivation
of the LDG method for the Poisson problem; the main difference is that,
as the name suggests, the LCDG method uses continuous finite element spaces.
\end{remark}

\begin{theorem}\label{BiharThm}
Suppose that $\eta_1>0$ and $\eta_2$ is non-negative.
Then there exists a unique $u_h\in V^h_r$ satisfying \eqref{Dbiharmonic}.
\end{theorem}

\begin{proof}
Since the problem is finite dimensional and linear, it suffices
to show that if $f=0$, $g=0$ and $q=0$, then the solution
is  identically zero.

To this end, we set $\bq_h = \nab_{h,0} u_h$, $v_h = \Div_{h,0} q_h$, and $\bz_h = \nab_h v_h$
so that $\Div_h \bz_h +j_{h,0}(u_h)+r_{h,0}(u_h) =0$.
To ease notation, we define the bilinear forms 
\begin{align*}
b(\bmu_h,\varphi_h) &:= -\big(\Div \bmu_h,\varphi_h\big)_{\mct} 
+\bl [\bmu_h]\cdot n,\{\varphi_h\}\br_{\mce},\\
b_I(\bmu_h,\varphi_h) &:= -\big(\Div\bmu_h,\varphi_h\big)_{\mct}
+\bl [\bmu_h]\cdot n,\{\varphi_h\}\br_{\mce^I},\\
c(\psi_h,\varphi_h) &: = \big(j_{h,0}(\psi_h),\varphi_h\big)_{\mct}+\big(r_{h,0}(\psi_h),\varphi_h\big)_{\mct}.
\end{align*}
It is then easy to verify that $(\Div_h \bmu_h,\varphi_h)_{\mct} = -b_I(\bmu_h,\varphi_h)$
and $\big(\Div_{h,0} \bmu_h,\varphi_h\big)_{\mct} = -b(\bmu_h,\varphi_h)$
for all $\bmu_h\in \bV_r^h$ and $\psi_h\in V^h_r$.
Furthermore, by Theorem \ref{SecondCorollary}, we have
$\big(\bmu_h,\nab_h \psi_h\big)_{\mct} = b(\bmu_h,\psi_h)$ and $
\big(\bmu_h,\nab_{h,0} \psi_h\big)_{\mct} = b_I(\bmu_h,\psi_h)$.
It then follows that we may write \eqref{Dbiharmonic} in the following mixed-form:
\begin{subequations}
\label{DBiharMixed}
\begin{alignat}{2}
\label{DBiharMixedA}
(\bq_h,\bmu_h)_{\mct}-b_I(\bmu_h,u_h) &= 0\qquad &&\forall \bmu_h\in \bV_r^h,\\
\label{DBiharMixedB}
(v_h,\psi_h)_{\mct}+b(\bq_h,\psi_h) &= 0\qquad &&\forall \psi_h \in V_r^h,\\
\label{DBiharMixedC}
(\bz_h, \btau_h)_{\mct}-b(\btau_h,v_h)& =0\qquad &&\forall \btau_h\in \bV_r^h,\\
\label{DBiharMixedD}
-b_I(\bz_h,\varphi_h)+c(u_h,\varphi_h) &=0\qquad &&\forall \varphi_h\in V_r^h.
\end{alignat}
\end{subequations}
Setting $\bmu_h = \bz_h$ in \eqref{DBiharMixedA} and $\varphi_h = u_h$
in \eqref{DBiharMixedD}, we have
\begin{align*}
(\bq_h,\bz_h)_{\mct}-b_I(\bz_h,u_h) &= 0,\quad 
\text{and}\quad
-b_I(\bz_h,u_h)+c(u_h,u_h) =0.
\end{align*}
Therefore, subtracting the two equations we get
$(\bq_h,\bz_h)_{\mct}-c(u_h,u_h) =0$.  Next, 
we set $\btau_h = \bq_h$ in \eqref{DBiharMixedC}
and $\psi_h = v_h$ in \eqref{DBiharMixedB} to obtain
\begin{align*}
(\bz_h, \bq_h)_{\mct}-b(\bq_h,v_h) =0,\quad \text{and}\quad
(v_h,v_h)_{\mct}+b(\bq_h,v_h) = 0.
\end{align*}
Adding the two equations yields $\|v_h\|_{L^2(\Ome)}^2 = -(\bz_h,\bq_h)_{\mct} = -c(u_h,u_h)\le 0$.
Therefore, $v_h\equiv 0$ and $c(u_h,u_h)\equiv 0$.  In particular, $u_h$ vanishes
on all of the boundary edges.  Since $v_h \equiv 0$, we also have $\bz\equiv 0$
by \eqref{DBiharMixedC}.  Setting $\bmu_h = \bq_h$ in \eqref{DBiharMixedA}, and
$\psi_h = u_h$ in \eqref{DBiharMixedB}, we have
\begin{align*}
\|\bq_h\|_{L^2(\Ome)}^2 -b_I(\bq_h,u_h) & = 0,\\
b(\bq_h,u_h) & = 0.
\end{align*}
Since $u_h$ vanishes on the boundary edges, $b(\bq_h,u_h) = b_I(\bq_h,u_h)$.
It then easily follows that $\bq_h \equiv 0$.  Finally, we have $b_I(\bmu_h,u_h) =0$
for all $\bmu_h\in \bV_r^h$.  This in turn implies that 
\begin{align*}
0 = (\bmu_h,\nab u_h)_{\mct} -\bl \{\bmu_h\}\cdot n,[u_h]\br_{\mce} = (\bmu_h,\nab u_h)_{\mct}\quad \forall \bmu_h\in \bV_r^h.
\end{align*}
Therefore, $\nab u_h|_K =0$ on all $K\in \mct$.  Since $[u_h]|_e =0$ across all edges, we
conclude that $u_h\equiv 0$.
\end{proof}

\begin{remarks}\
\begin{enumerate}
\item[(a)] 
To obtain optimal order error estimates, we expect that the penalty
parameters must scale like $\eta_1 = \mathcal{O}(h^{-3})$ and $\eta_2 = \mathcal{O}(h^{-1})$.

\item[(b)] 
The construction of DG schemes
with other types of boundary conditions can easily
be constructed by specifying the boundary data
to different discrete differential operators.  
For example, if simply supported plate
boundary conditions $u=g$ and $\Del u = q$
are provided, then the corresponding discrete biharmonic operator is
$\Del_{h,q}\Del_{h,g}= \Div_h \nab_{h,q} \Div_h \nab_{h,g} u_h$.
On the other hand, if Cahn-Hilliard-type boundary conditions
$\p u/\p n = g$ and $\p\Del u/\p n=q$ are given, then the discrete
biharmonic operator is $\Div_{h,{\bm q}} \nab_h \Div_{h,{\bm g}} \nab_h u_h$,
where ${\bm q} = q n $ and ${\bm g} = g n$.
\end{enumerate}
\end{remarks}

\subsection{Quasi-linear second order PDEs}\label{sec-6.3}
\subsubsection{\bf The $p$-Laplace equation}
We now extend the discrete differential framework to 
some non-linear elliptic problems.  Although we can 
formulate the method for a very general class
of quasi-linear PDEs, we shall focus our attention
to a prototypical example, the $p$-Laplace equation\ $(2\le p<\infty)$:
\begin{subequations}
\label{pLaplace}
\begin{alignat}{2}
-\Div \big(|\nab u|^{p-2}\nab u\big) & = f\qquad &&\text{in }\Ome,\\
u & = g\qquad &&\text{on }\partial \Ome.
\end{alignat}
\end{subequations}

Similar to the Poisson problem, the DG method for \eqref{pLaplace}
is obtained by simply replacing the {\em grad} and {\em div} operators by
their respective discrete versions and adding a stability term
to ensure that the resulting bilinear form is coercive over $V^h_r$.
In this case, the discrete problem reads: find $u_h\in V^h_r$ satisfying
\begin{align}\label{DpLaplace}
-\Div_h \big(|\nab_{h,g} u_h|^{p-2}\nab_{h,g} u_h\big) +j^{(p)}_{h,g}(u_h)  & = \cP^h_r f,
\end{align}
where $j^{(p)}_{h,g}(\cdot)$ is defined by 
\begin{align}
\label{jphDef}
\big(j^{(p)}_{h,g}(v),\varphi_h)_{\mct} 
= \bl \eta_1 \big|[v]\big|^{p-2} [v],[\varphi_h]\br_{\mce^I}
+\bl \eta_1 |v|^{p-2} (v-g),\varphi_h\br_{\mce^B}  
\end{align}
for all $\varphi_h\in V^h_r$.  Here, $\eta_1>0$ is a penalization parameter.
Similar to the discrete Poisson problem, the discretization has several interpretations.
By  the definition of the discrete divergence and gradient operators, we can write \eqref{DpLaplace}
in the mixed formulation
\begin{align}\label{pLaplaceMixed}
\big(|\bq_h|^{p-2}\bq_h,\nab \varphi_h\big)_{\mct} 
&-\Bl \{|\bq_h|^{p-2}\bq_h\}\cdot n,[\varphi_h]\Br_{\mce} 
+\big(j_{h,g}^{(p)}(u_h),\varphi_h\big)_{\mct} \\
&\qquad \nonumber 
= (f,\varphi_h)_{\mct} \qquad\forall \varphi_h\in V^h_r, 
\end{align}
where $\bq\in \bV^h_r$ satisfies \eqref{dualmixed1}.  We emphasize that the gradient appearing 
in the left-hand side of equation \eqref{pLaplaceMixed} is the piecewise gradient.

\begin{remark}
In \cite{BurmanErn08}, Burman and Ern proposed and analyzed
the following LDG method for the $p$-Laplace equation (with $r=1$ and $g=0$):
\begin{align}\label{BurmanErn}
\big(|\nab_{h,0} u_h|^{p-2} \nab_{h,0} u_h ,\nab_{h,0} \varphi_h\big)_{\mct} 
+ \bl \eta_1 \big|[u_h]\big|^{p-2}[u_h],[\varphi_h]\br_{\mce}
= (f,\varphi_h)_{\mct}  
\end{align}
for all $\varphi_h\in V^h_r$. Here, the authors showed the existence and uniqueness
of the DG method \eqref{BurmanErn} provided $\eta_1 = \mathcal{O}(h^{1-p})$.
In addition, Burman and Ern showed that 
the  approximate solutions converge to $u$ strongly in $L^p(\Ome)$,
and $\nab_{h,0} u_h$ converges
to $\nab u$ strongly in $\bL^p(\Ome)$.  Furthermore, in 
the one dimensional setting, numerical experiments indicate
a convergence rate of at least $h^{3/4}$ for $p\in \{3,4,5\}$ and smooth exact
solution.

Clearly, method \eqref{BurmanErn} has a similar structure to \eqref{pLaplaceMixed}, but they
are different methods when $p \neq 2$. Indeed, since $|\nab_{h,g} u_h|^{p-2}\nab_{h,g} u_h \not\in \bV_r^h$, 
we cannot use Theorems {\rm \ref{FirstCorollary}--\ref{SecondCorollary}} and simply write
\begin{align*}
-\big(\Div_h (|\nab_{h,g} u_h|^{p-2}\nab_{h,g} u_h),\varphi_h\big)_{\mct}
 = \big( |\nab_{h,g} u_h|^{p-2}\nab_{h,g} u_h,\nab_{h,0} \varphi_h\big)_{\mct}.
\end{align*}
In the following section, we show by way of numerical experiments
that the DG method \eqref{pLaplaceMixed} converges with optimal order
provided the exact solution is sufficiently smooth.
\end{remark}
 
\subsubsection{\bf Numerical experiments of the $p$-Laplace equation}

In this subsection we perform some numerical experiments to gauge
the effectiveness of the DG method \eqref{DpLaplace}. 
We take the domain to be the unit square $\Ome = (0,1)^2$
and choose the data $f$ such that the exact solution 
is $u = \sin(\pi x_1)\sin(\pi x_2)$ and $p=5$.  In all numerical experiments, 
we take the penalty parameter to be 
$\eta_1 = 20/h^{p-1} = 20/h^4$.

The resulting errors in the cases $r=1$ and $r=2$ are recorded
in Table \ref{pLaplaceTable}.  The table clearly suggests
that the following rates of convergence hold for smooth test problems:
\begin{align*}
\|u-u_h\|_{L^2(\Ome)} = \mathcal{O}(h^{r+1}),\qquad \|\nab u-\nab_h u_h\|_{L^2(\Ome)} = \mathcal{O}(h^r).
\end{align*}

\begin{table}[ht]
{\small \caption{\label{pLaplaceTable}The errors of the computed solution 
and rates of convergence of the DG method \eqref{DpLaplace}
with solution $u=\sin(\pi x_1)\sin(\pi x_2)$ on the unit square with $p=5$
and $\eta = 20/h^4$.}
\begin{tabular}{|c|c|c|c|c|c|}
\hline
& $h$ & $\|u-u_h \|_{L^2}$ & order & $\|\nab u-\nab_h u_h\|_{L^2}$ & order\\
\hline
$r=1$
&1.00E-01	&6.48E-03	&		&2.33E-01	&\\
&5.00E-02	&1.08E-03	&2.59	&1.16E-01	&1.01\\
&2.50E-02	&1.95E-04	&2.47	&5.86E-02	&0.99\\
&1.25E-02	&4.99E-05	&1.97	&2.94E-02	&0.99\\
\hline
$r=2$
&1.00E-01	&2.25E-03	&		&1.36E-02	&\\
&5.00E-02	&2.83E-04	&2.99	&3.01E-03	&2.18\\
&2.50E-02	&3.60E-05	&2.97	&7.30E-04	&2.04\\
&1.25E-02	&4.51E-06	&3.00	&1.80E-04	&2.02\\
\hline
\end{tabular}
}
\end{table}

\subsection{Second order linear elliptic PDEs in non-divergence form}

As a fourth example, we consider second order elliptic PDEs
written in non-divergence form.  Namely, we consider the problem 
of finding a strong solution satisfying
\begin{subequations}
\label{NonVar}
\begin{alignat}{2}
\label{NonVar1}
-\tilde{\bA}:D^2 u & = f\qquad &&\text{in }\Ome,\\
\label{NonVar2}
u & = g\qquad &&\text{on }\p\Ome.
\end{alignat}
\end{subequations}
Here, $\tilde{\bA}\in [C^{0,\alpha}(\Ome)]^{d\times d}\ (\alpha\in (0,1))$ is a given
positive definite matrix, and $\tilde{\bA}:D^2 u$ denotes
the Frobenius inner product, i.e., 
$\tilde{\bA}:D^2 u = \sum_{i,j=1}^d A_{i,j} \frac{\p^2 u}{\p x_i\p x_j}$.

We note that if $\tilde{\bA}$ is sufficient smooth, e.g. 
if $\Div \tilde{\bA}\in \bL^\infty(\Ome)$, 
then we may write the PDE \eqref{NonVar1} as 
$-\Div (\tilde{\bA}\nab u) +(\Div \tilde{\bA})\cdot \nab u = f$.  We can then apply
any of the standard numerical methods for convection-diffusion
equations to problem \eqref{NonVar}.  On the other hand,
if $\tilde{\bm A}$ only has the regularity 
$\tilde{\bA}\in [C^{0,\alpha}]^{d\times d}$, then this argument
fails and the construction of numerical methods is less obvious.
As far as we are aware, only two finite element methods 
have appeared in the literature that addressed the numerical
approximation of problems such as \eqref{NonVar}.
In \cite{JensenSmears12a}  
Jensen and Smears propose a $\pol_1$ finite element method
for the Hamilton-Jacobi-Bellman equation.  To handle
the lack of regularity of the coefficient matrix, the authors
``freeze the coefficients'' element-wise, and then perform 
the usual integration-by-parts technique.  By modifying
the framework of Barles-Sougandidis \cite{BarlesSoug91}, Jensen
and Smears show that the numerical solutions converge
to the exact solution strongly in $H^1$. 
Another finite element method for problem \eqref{NonVar},
which is closely related to ours, is the one proposed by
Lakkis and Pryer in \cite{LakkisPryer12}.  Here, the authors
used the notation of a {\em discrete Hessian} and rewrite
problem \eqref{NonVar} in a mixed form.
The advantage of their approach is that the finite element spaces
are simply the globally continuous Lagrange elements, which are simple
to implement.  A possible disadvantage of their approach is
that the notion of their discrete Hessian is not local, and therefore
writing the problem in its primal form results in a dense stiffness matrix.

To formulate the DG method for the PDE using
the discrete differential calculus framework, we again
replace the continuous differential operators by 
the discrete ones.  In addition, we have to project 
both sides of the equation onto the finite element
space.  This then leads to the following problem:
find a function $u_h\in V^h_r$ satisfying
\begin{align}\label{DNonVar}
-\cP^r_h (\tilde{\bA}:D^2_{h,g} u_h) +j_{h,g}(u_h)= \cP^r_h f
\end{align}
with $j_{h,g}(\cdot)$ defined by \eqref{jhDef}.
Equivalently, the DG method is to find $u_h \in V_r^h$ such that
\begin{align*}
-\cP_r^h (\tilde{\bA}:\tilde{\bm r}_h) +j_{h,g}(u_h)= \cP_h^r f,
\end{align*}
where ${\bm r}_h\in \widetilde{\bV}_h$ satisfies 
\eqref{dualmixed1}--\eqref{dualmixed2}.

\begin{remark}
If the coefficient matrix $\tilde{\bA}$ is constant, then the DG method \eqref{DNonVar}
reduces to
the LDG method for 
the PDE $-\Div (\tilde{\bA}\nab u) = f$ with appropriate
boundary conditions.
\end{remark}

Below we present some numerical test results on the DG method \eqref{DNonVar}
with the following parameters: $\Ome = (-0.5,0.5)^2$, $f=0$ and 
\begin{align} \label{NonDivData}
g = \left\{\begin{array}{cc}
x_1^{-4/3} & \text{if } x_2=0,\\
-x_2^{4/3} & \text{ if }x_1=0,\\
x_1^{4/3}-1 & \text{if }x_2=1,\\
1-x_2^{4/3} & \text{if }x_1=1
\end{array}\right.,
\quad \tilde{\bm A}
=  \frac{16}9\begin{pmatrix}
x_1^{2/3} & -x_1^{1/3} x_2^{1/3}\smallskip\\
-x_1^{-1/3} x_2^{1/3} & x_2^{2/3}
\end{pmatrix}.
\end{align}
It can readily be checked that the exact solution
is given by $u = x_1^{4/3}-x_2^{4/3}\in C^{1,\frac13}(\overline{\Ome})$.
We note that this is a particularly challenging example since 
$\tilde{\bm A}$ is not uniformly elliptic.  
The resulting errors for decreasing values of $h$ are listed
in Table \ref{NonDivTable}, and a computed 
solution and error is depicted in Figure \ref{NonDivFigure}.  
The table clearly indicates the convergence
of the method, although the exact rates of convergence are not obvious.

\begin{table}[ht]
{\small
 \caption{\label{NonDivTable}The errors of the computed solution 
and rates of convergence with $r=1$.}
   {
   \begin{tabular}{|c|c|c|c|c|}
   \hline
   $h$ & $\|u-u_h \|_{L^2}$ & order & $\|\nab u-\nab_h u_h\|_{L^2}$ & order\\
   \hline
1.00E-01		&5.17E-03	&		&1.45E-01	&\\
5.00E-02		&3.49E-03	&0.56	&9.52E-02	&0.60\\
2.50E-02	&2.59E-03	&0.43	&6.50E-02	&0.55\\
1.25E-02	&2.08E-03	&0.32	&4.81E-02	&0.43\\
\hline
\end{tabular}
}
}
\end{table}

\begin{figure}[htbp] 
   \centering
   \includegraphics[scale=0.14]{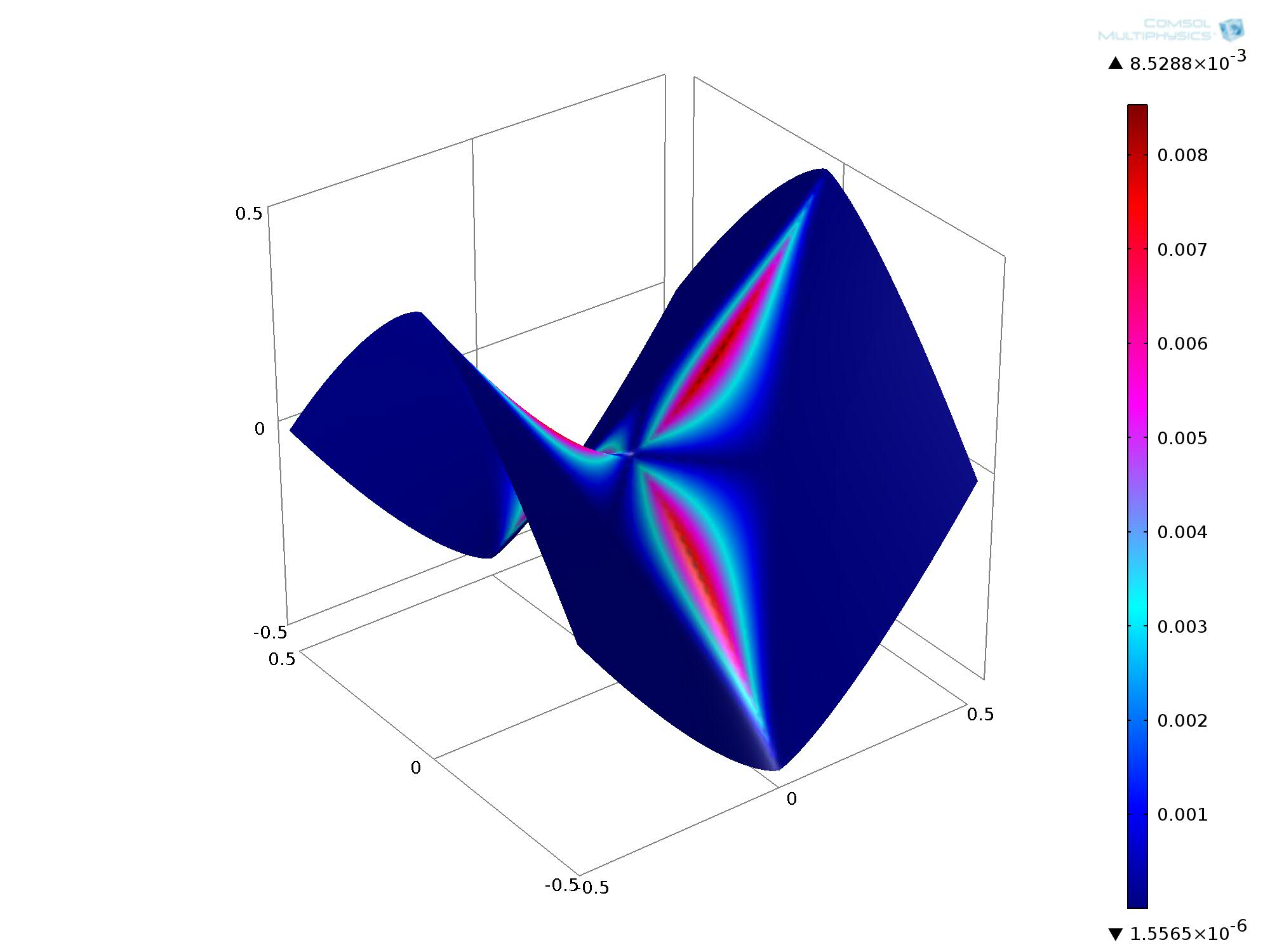} 
   \caption{Computed solution (height) and error (surface)
   of the DG method \eqref{DNonVar} with data \eqref{NonDivData} and parameters $\eta = 20$, $r=1$ and $h=0.0125$.}
   \label{NonDivFigure}
\end{figure}

\subsection{Fully nonlinear time dependent first order PDEs} \label{HJ}
As a fifth example, we consider Hamilton-Jacobi equations.
Namely, we consider the problem of finding the viscosity solution
$u \in \mathcal{A} \subset C^0 \bigl(\Omega \times (0,T]\bigr)$
for the PDE problem
\begin{subequations} \label{nonlinear1}
\begin{alignat}{2}
\label{nonlinear1_1}
u_t + H \left( \nabla u \right) & = 0 \qquad &&\text{in }\Ome\times (0,T],\\
\label{nonlinear1_2}
u & = g\qquad &&\text{on } \Gamma \subset \p\Ome\times (0,T], \\ 
\label{nonlinear1_3}
u & = u_0 \qquad &&\text{on } \Omega \times \{ 0\},
\end{alignat}
\end{subequations}
where the operator $H$ is a continuous and possibly nonlinear function 
and $\mathcal{A}$ is a function class in which the viscosity solution $u$ is 
unique.  We note that the following scheme can also be 
adapted for $H$ a function of $u$, $x$, and $t$.

Let $B(\Ome)$, $BUC(\Ome)$, $USC(\Ome)$ and $LSC(\Ome)$ denote, respectively, 
the spaces of bounded, bounded uniformly continuous, upper semi-continuous 
and lower semicontinuous functions on $\Ome$.
We now recall the well-known existence and uniqueness theorem for the 
corresponding Cauchy problem which was first proved in \cite{Crandall_Lions84}.

\begin{theorem} \label{CL84}
Let $H \in C({\bf R}^d)$, $u_0 \in BUC(\bf{R}^d)$.  Then there is exactly one function
$u \in BUC({\bf R}^d \times [0,T])$ for all $T>0$ such that $u(x,0) - u_0(x)$, and for 
every $\phi \in C^1\left( {\bf R}^d \times (0,\infty) \right)$ and $T>0$, 
if $(x_0, t_0)$ is a local maximum (resp. local minimum) point of $u-\phi$ on ${\bf R}^d \times (0,T]$, 
then
\[
\phi_t (x_0, t_0) + H \left( \nabla \phi (x_0, t_0) \right) \leq 0
\] 
( resp. $\phi_t (x_0, t_0) + H \left( \nabla \phi (x_0, t_0) \right) \geq 0$).
\end{theorem}

\begin{definition}
The function $u$ whose existence and uniqueness is guaranteed by 
Theorem~{\rm\ref{CL84}} is called the {\em viscosity solution} of the 
Cauchy version of \eqref{nonlinear1}.
\end{definition}

Recently, a nonstandard LDG method was proposed by Yan and Osher 
in \cite{OsherYan11} for approximating the viscosity solution of 
the Hamilton-Jacobi problem \eqref{nonlinear1}.  
The main idea of \cite{OsherYan11} is to approximate the ``left" and ``right"
side derivatives of the viscosity solution and to judiciously combine them 
through a monotone and consistent numerical Hamiltonian (cf. \cite{Shu07})
such as the Lax-Friedrichs numerical Hamiltonian
\begin{equation} \label{LF1}
\hH({{\bq}}^-, {{\bq}}^+) : = H \left( \frac{{{\bq}}^- + {{\bq}}^+}{2} \right)
- \frac{1}{2} \boldsymbol{\beta} \cdot \left(  {\bq}^+ - {\bq}^- \right) , 
\end{equation} 
where $\boldsymbol{\beta} \in \bfR^d$ is an undetermined nonnegative vector 
chosen to enforce the monotonicity property of $\hH$, or the Godunov 
numerical Hamiltonian
\begin{equation} \label{godunov1}
\hH({{\bq}}^-, {{\bq}}^+) : = \ext_{q_1 \in I \left( q^-_1, q^+_1 \right)} \cdots
\ext_{q_d \in I \left( q^-_d, q^+_d \right)} H({\bq}),
\end{equation}
where
\[
	\ext_{v \in I \left( a, b \right)} := \begin{cases}
	\min_{a \leq v \leq b}, & \text{if } a \leq b , \\
	\max_{b \leq v \leq a}, & \text{if } a > b , 
	\end{cases}
\]
for $I(a,b) := \bigl[ \min(a,b) , \max(a,b) \bigr]$. 

\begin{remark}
For $r=0$ and $\boldsymbol{\beta} = \mathbf{1}$ on a uniform rectangular grid, 
the second term in \eqref{LF1} is equivalent to a second order finite 
difference approximation for the negative Laplacian operator scaled by $h$.
Thus, the second term in \eqref{LF1} is called a {\em numerical viscosity}, 
and the method is a direct realization of the vanishing viscosity method from 
PDE theory.  However, for high order elements and variable coefficient 
vector $\boldsymbol{\beta}$, while we do not exactly recover a scaled 
Laplacian operator, we do recover some of the stabilizing properties from 
adding a second-order-like perturbation.
\end{remark}

With the correct choice of discrete derivatives, we can rewrite the
nonstandard LDG method of Yan and Osher \cite{OsherYan11} as follows: 
For each time step $n=1,2,3,\ldots$ with $u_h^0 
:= {\cP}_r^h u_0$, find $u^n_h \in V_r^h$ using the recursive relation
\[
u_h^n = u_h^{n-1} + \Delta t\, {\cP}_r^h \hH \bigl( \nabla_{h,g}^- u_h^{n-1}, 
\nabla_{h,g}^+ u_h^{n-1} \bigr).
\]
In addition to the above Euler time-stepping method with $\hH$ given by 
the Lax-Friedrichs numerical Hamiltonian defined in \eqref{LF1}, 
in \cite{OsherYan11} Yan and Osher also implemented the explicit 
third-order TVD Runge-Kutta time-stepping method given in \cite{Shu_Osher88}. 
Tests include one and two dimensional problems for $r \geq 0$.

\begin{remark}
For $r=0$ on a Cartesian grid, the method reduces to the convergent FD method of Crandall and 
Lions proposed and analyzed in \cite{Crandall_Lions_84}. 
\end{remark}

\subsection{Fully nonlinear second order PDEs} \label{sec-6.6}
As the last example, we consider fully nonlinear second order elliptic PDEs. 
Namely, we consider the problem of finding the viscosity solution 
$u \in \mathcal{A} \subset C^0 (\Omega)$ for the PDE problem
\begin{subequations} \label{nonlinear2}
\begin{alignat}{2} \label{nonlinear2_1}
F [ u ] := F \left( D^2 u , x \right) & = 0 
\qquad &&\text{in }\Ome,\\ 
\label{nonlinear2_2}
u & = g\qquad &&\text{on }\p\Ome,
\end{alignat}
\end{subequations}
where the operator $F$ can be nonlinear in all arguments and $\mathcal{A}$ 
is a function class in which the viscosity solution is unique. Throughout this 
subsection, we assume that $F [v]$ is elliptic for all $v \in \mathcal{A}$, 
$F$ satisfies a comparison principle and problem~\eqref{nonlinear2} has a 
unique viscosity solution $u \in \mathcal{A}$.  The definitions for the 
above terms are given below. We again use the same function and space notations 
from Section~\ref{HJ}. 
We also note that the following scheme can also be adapted for 
$F$ a function of $u$ and $\nabla u$.

For ease of presentation, we write \eqref{nonlinear2} as
\begin{equation} \label{nonlin2}
F(D^2u, x) = 0 \qquad\mbox{in } \oOme, 
\end{equation}
where we have used the convention of writing the boundary condition as a
discontinuity of the PDE (cf. \cite[p.274]{BarlesSoug91}). 
Also for ease of presentation, we assume that $F$ is a continuous function 
and refer the reader to \cite{BarlesSoug91, Feng_Lewis12, Feng_Kao_Lewis11} 
for the case when $F$ is only a bounded function. 

The following three definitions are standard and can be found 
in \cite{Gilbarg_Trudinger01, Caffarelli_Cabre95,BarlesSoug91}.

\begin{definition}
Equation \eqref{nonlin2} is said to be elliptic if for all 
$x\in \oOme$ there holds
\begin{align}\label{elliptic_def}
F(\tilde{\bA}, x) \leq F(\tilde{\bm B}, x) \qquad\forall 
\tilde{\bA},\tilde{\bm B}\in \cS^{d\times d},\, \tilde{\bA}\geq \tilde{\bm B}, 
\end{align}
where $\tilde{\bA}\geq \tilde{\bm B}$ means that $\tilde{\bA}-\tilde{\bm B}$ 
is a nonnegative definite matrix, and $\cS^{d\times d}$ denotes
the set of real symmetric $d\times d$ matrices.
\end{definition}
We note that when $F$ is differentiable, ellipticity can also be 
defined by requiring that the matrix $\frac{\partial F}{\partial D^2 u}$ 
is negative semi-definite (cf. \cite[p. 441]{Gilbarg_Trudinger01}).

\begin{definition}\label{viscosity_def}
A function $u\in C^0(\Ome)$ is called a viscosity subsolution (resp. 
supersolution) of \eqref{nonlin2} if, for all $\varphi\in C^2(\oOme)$,  
if $u-\varphi$ (resp. $u-\varphi$) has a local maximum 
(resp. minimum) at $x_0\in \oOme$, then we have
\[
F(D^2\varphi(x_0),x_0) \leq 0 
\]
(resp. $F(D^2\varphi(x_0), x_0) \geq 0$).
The function $u$ is said to be a viscosity solution of \eqref{nonlin2}
if it is simultaneously a viscosity subsolution and a viscosity
supersolution of \eqref{nonlinear2}.
\end{definition}

\begin{definition}\label{def3.2}
Problem \eqref{nonlin2} is said to satisfy a {\em comparison 
principle} if the following statement holds. For any upper semi-continuous 
function $u$ and lower semi-continuous function $v$ on $\overline{\Omega}$,  
if $u$ is a viscosity subsolution and $v$ is a viscosity supersolution 
of \eqref{nonlin2}, then $u\leq v$ on $\overline{\Omega}$.
\end{definition}

Inspired by the work of Yan and Osher \cite{OsherYan11}, 
the first and second authors of this paper recently proposed in 
\cite{Feng_Lewis12} a nonstandard LDG method for approximating the 
viscosity solution of the fully nonlinear second order problem \eqref{nonlinear2}
in one-dimension.
The main idea of \cite{Feng_Lewis12} is to use all four of the various ``sided" approximations
for the second order derivative, \eqref{s6.6_hessians}, of the viscosity solution
and to judiciously combine them through a g-monotone (generalized
monotone) and consistent numerical operator such as the following 
Lax-Friedrichs-like numerical operator that has been adopted from \cite{Feng_Lewis12} for an 
arbitrary dimensional problem:
\begin{align}\label{LF2}
\hF(\tilde{\bP}^{- -}, \tilde{\bP}^{- +}, \tilde{\bP}^{+ -}, \tilde{\bP}^{+ +}, \xi)  
&:= F\Bigl( \frac{\tilde{\bP}^{- +}+\tilde{\bP}^{+ -}}{2}, \xi \Bigr) \\
&\qquad 
+ \tilde{\bA} : \bigl(\tilde{\bP}^{- -}-\tilde{\bP}^{- +}-\tilde{\bP}^{+ -}+\tilde{\bP}^{+ +} \bigr), \nonumber 
\end{align}
where $\tilde{\bA} \in \bfR^{d \times d}$ is a nonnegative constant matrix that is 
chosen to enforce the g-monotonicity property of $\hF$. The consistency 
of $\hF$ is defined by fulfilling the following property: 
\[
\hF(\tilde{\bP},\tilde{\bP},\tilde{\bP}, \tilde{\bP} , x)
=F(\tilde{\bP}, x) \quad\forall \tilde{\bP} \in \bfR^{d \times d}, x \in \Omega;
\]
and the g-monotonicity requires that $\hF$ is monotone increasing in its
first and fourth arguments (i.e., $\tilde{\bP}^{- -}$, $\tilde{\bP}^{+ +}$) and monotone 
decreasing in its second and third arguments (i.e., $\tilde{\bP}^{- +}$, $\tilde{\bP}^{+ -}$).

Below we give a reformulation of the nonstandard LDG method of \cite{Feng_Lewis12}
using our DG finite element differential calculus machinery. 
To this end, we simply replace the continuous differential operators 
by multiple copies of discrete ones. Due to the lack of integration 
by parts caused by the nonlinearity, we have to project the equation 
onto the DG finite element space.  This then leads to the following scheme 
of finding $u_h \in V_r^h$ such that
\begin{equation} \label{nonlin2_dg}
{\cP}_r^h \, \hF \left( D_{h,g}^{- -} u_h, D_{h,g}^{- +} u_h , D_{h,g}^{+ -} u_h, D_{h,g}^{+ +} u_h ,  x \right) = 0,
\end{equation}
where $D_{h,g}^{- -}$, $D_{h,g}^{- +}$, $D_{h,g}^{+ -}$, $D_{h,g}^{+ +}$ are the four
sided numerical Hessians (with the prescribed boundary data $g$)
defined in Section \ref{sec-3}.
It can be shown that \eqref{nonlin2_dg} is indeed equivalent to 
the LDG scheme of \cite{Feng_Lewis12} in one-dimension.  

\begin{remarks}\
\begin{enumerate}
\item[(a)] When $r=0$ and $d=1$, scheme \eqref{nonlin2_dg} reduces to the  
FD method given in \cite{Feng_Kao_Lewis11}, which was proved to be 
convergent.

\item[(b)] For $r=0$ and $\tilde{\bA} = \tilde{\bm I}_{d\times d}$ 
on a rectangular grid, the second term on the right hand side in \eqref{LF2}
is equivalent to a second order finite difference approximation 
for the biharmonic operator scaled by $h^2$. Thus, the second term
in \eqref{LF2} is called a {\em numerical moment}, 
and the method is a direct 
realization of the vanishing moment method proposed 
in  \cite{FengNeilan09a,FengNeilan09b}.
However, for high order elements and variable coefficient matrices $\tilde{\bA}$, 
while we do not exactly recover a scaled biharmonic operator, we do recover 
some of the stabilizing properties from adding a fourth-order-like perturbation.

\item[(c)] Numerical tests of \cite{Feng_Lewis12} show the above discretization 
can eliminate many, and in some cases all, of the numerical artifacts 
that plague standard discretizations for fully nonlinear second order PDEs
(cf. \cite{Feng_Glowinski_Neilan12} and the references therein).

\item[(d)] Fully discrete schemes of \cite{Feng_Lewis12} for parabolic fully 
nonlinear second order PDEs can also be recast using the DG finite
element differential calculus machinery.

\end{enumerate}
\end{remarks}

To solve the algebraic system \eqref{nonlin2_dg}, a nonlinear solver must be used.  
Numerical tests of \cite{Feng_Lewis12} show that when the initial guess 
for $u_h$ is not too close to a non-viscosity solution of the PDE problem, 
a Newton-based solver performs well as long as the 
solution is not on the boundary of the admissible set $\mathcal{A}$.  
However, in the degenerate case, a split solver based on the DWDG 
discretization for the Poisson equation from Section~\ref{dwdg} and the 
Lax-Friedrichs-like operators in \eqref{LF2} appear to be 
better suited. This new solver for \eqref{nonlin2_dg} is given 
below in Algorithm~\ref{split_dwdg}. We note that this solver 
appears to work well even for some cases when the initial guess 
for $u_h$ is not in $\mathcal{A}$. Thus, the solver uses key tools from 
the discretization to address the issue of conditional uniqueness 
of viscosity solutions.

\begin{algorithm} \label{split_dwdg}

Pick $u_h^{(0)} \in V_r^h$.  Let $\Lambda_{h,g}^{+ -} v$ and $\Lambda_{h,g}^{- +} v$ 
denote the diagonal matrices formed by the diagonals of $D_{h,g}^{+ -} v$ and $D_{h,g}^{-+} v$, respectively, 
and $\Lambda_{h,g}^{- -} v$ and $\Lambda_{h,g}^{+ +} v$ denote the diagonal matrices 
formed by the diagonals of $D_{h,g}^{- -} v$ and $D_{h,g}^{+ +} v$, respectively, for
all $v \in V_r^h$. Let $\widetilde{\boldsymbol{\lambda} \mathbf{I}}$ denote the 
diagonal matrix formed by the vector
$\boldsymbol{\lambda} \in \mathbf{R}^d$. \\ 
For $n = 1,2 ,3, \ldots$, 
\begin{enumerate}
\item[{\rm Step 1:}] For $i=1,2,\ldots, d$, set
\begin{align*}
[G]_i^{(n)} &:= F\Bigl( \bigl((D_{h,g}^{- +}-\Lambda_{h,g}^{- +})/2+ 
(D_{h,g}^{+ -}-\Lambda_{h,g}^{+ -})/2 \bigr) u_h^{(n-1)} 
+\widetilde{\boldsymbol{\lambda}^{(n)}\mathbf{I}},x\Bigr)\\
&\qquad 
+A \Bigl[\Lambda_{h,g}^{- -} u_h^{(n-1)} + \Lambda_{h,g}^{+ +} u_h^{(n-1)} 
- 2 \, \widetilde{\boldsymbol{\lambda}^{(n)}\mathbf{I}} \Bigr]_{ii}
\end{align*}
for a fixed constant ${A} > {0}$ such that $G^{(n)}$ is monotone decreasing 
with respect to $\boldsymbol{\lambda}^{(n)}$.

\item[{\rm Step 2:}] Solve for $\boldsymbol{\lambda}^{(n)} \in \bV_r^h$
such that
\[
\Bigl( [G]_i^{(n)} , \phi_i \Bigr)_{\mathcal{T}_h} = 0 
\]
for all $\phi_i \in V_r^h$,  $i=1,2,\ldots,d$.

\item[{\rm Step 3:}] Solve for $u_h^{(n)}\in V^h_r$ such that
\[
\frac12 \bigl(\nab_{h,g}^+ u_h^{(n)},\nab_{h,0}^+ v_h\bigr)_{\mct}
+\frac12\bigl(\nab_{h,g}^- u_h^{(n)},\nab_{h,0}^- v_h\bigr)_{\mct}
=\sum_{i=1}^d \bigl(\lambda^{(n)}_i, v_h\bigr)_{\mct} 
\]
for all $v_h \in V_r^h$.
\end{enumerate}
Note that we are solving the discretization that results from the choice
$\tilde{\bA} = A\, \tilde{\bm I}_{d\times d}$ in \eqref{LF2}.

\end{algorithm}

The above solver is a fixed point method for the diagonal of the Hessian 
approximation formed by $\bigl( D_{h,g}^{- +} + D_{h,g}^{+ -}\bigr)/2$.
We can see that the nonlinear equation in Step 2 is entirely monotone 
and local in its nonlinear components when $A$ is sufficiently large.
Step 3 is well-defined due to the well-posedness of the DWDG method for 
Poisson's equation.
Lastly, as written, we can see that the numerical moment inspired by the 
vanishing moment methodology serves as the motivation for using 
the diagonal of the Hessian approximation as the fixed-point parameter.
Numerical tests indicate that the solver 
destabilizes numerical artifacts even when they exist (cf. \cite{Feng_Lewis12}).
Thus, the solver directly addresses the issue of conditional uniqueness by 
enforcing the preservation of monotonicity in the Hessian approximation at each iteration.

%

\end{document}